\UseRawInputEncoding
\pdfoutput=1

\documentclass[leqno,11pt]{amsart}

\usepackage{amsmath}

\usepackage{amssymb}
\usepackage{graphicx}

\usepackage{color}
\usepackage[dvipsnames]{xcolor}

\newtheorem{thm}{Theorem}[section]

\newtheorem{thmx}{Theorem} 

\newtheorem{prop}[thm]{Proposition}
\newtheorem{lemma}[thm]{Lemma}

\theoremstyle{definition}
\newtheorem{defn}[thm]{Definition}
\newtheorem{ex}[thm]{Example}

\theoremstyle{remark}
\newtheorem{remark}[thm]{Remark}

\numberwithin{equation}{section}

\def\EE{\mathcal{E}^{1,2}}
\def\SS{{{S}}^{2,2}}

\def\A{\mathrm{O}^{\uparrow}_{+}(2,3)}
\def\NN{\mathcal{N}_f}

\def\aa{\mathfrak{o}(2,3)}

\def\MM{\mathbb{R}^{2,3}}
\def\df{f^{\sharp}}

\def\R{\mathbb{R}}

\def\Z{\mathbb{Z}}

\def\SL{\mathrm{SL}}

\def\T{{^t\!}}

\def\ker{\mathrm{Ker\,}}
\def\ima{\mathrm{Im\,}}
\def\span{\mathrm{span\,}}
\def\tr{\mathrm{tr}}
\def\det{\mathrm{det}}

\DeclareMathOperator{\sgn}{sgn}

\begin{document}
\title[]{Conformal geometry of \\ quasi-umbilical timelike surfaces}

\author{Emilio Musso}
\address{(E. Musso) Dipartimento di Scienze Matematiche, Politecnico di Torino,
Corso Duca degli Abruz\-zi 24, I-10129 Torino, Italy}
\email{emilio.musso@polito.it}

\author{Lorenzo Nicolodi}
\address{(L. Nicolodi) Di\-par\-ti\-men\-to di Scienze Matematiche, Fisiche e Informatiche,
Uni\-ver\-si\-t\`a degli Studi di Parma, Parco Area delle Scienze 53/A,
I-43124 Parma, Italy}
\email{lorenzo.nicolodi@unipr.it}

\author{Mason Pember}
\address{(M. Pember) 
Department of Mathematics, University of Bath,
Claverton Down, Bath BA2 7AY, United Kingdom}
\email{mason.j.w.pember@bath.edu}

\thanks{Authors partially supported by PRIN 2017 and 2022 ``Real and Complex 
Manifolds: Topology, Geometry and Holomorphic Dynamics" (prott. 2017JZ2SW5-004 and 
2022AP8HZ9-003); and by the GNSAGA of INdAM. The present research was also partially 
supported by MIUR grant ``Dipartimenti di Eccellenza" 2018-2022, CUP: E11G18000350001, 
DISMA, Politecnico di Torino.}

\subjclass[2000]{53C50, 53A30, 53B30}

\date{Version of october 8, 2024}


\keywords{Conformal Lorentz geometry, timelike surfaces, quasi-umbilical surfaces, 
isothermic surfaces, harmonic maps.}

\begin{abstract} 
The paper focuses on the conformal Lorentz geometry of quasi-umbilical timelike surfaces in the $(1+2)$-Einstein 
universe, the conformal compactification of Minkowski 3-space 
realized as the space of oriented null lines through the origin of $\R^{2,3}$.
A timelike immersion of a surface $X$ in the Einstein universe is {quasi-umbilical} if its shape operator at 
any point of $X$ is non-diagonalizable over $\mathbb{C}$.
We prove that quasi-umbilical surfaces are isothermic, that their conformal deformations depend on one arbitrary function 
in one variable, and show that their conformal Gauss map is harmonic. 
We investigate their geometric structure and show how to construct all quasi-umbilical surfaces 
from null curves in the 4-dimensional neutral space form $S^{2,2}=\{x\in \mathbb{R}^{2,3} \mid \langle x,x\rangle=1 \}$.
\end{abstract}

\maketitle

\section{Introduction}\label{s:intro} 

This paper is a contribution to the subject of conformal Lorentz geometry 
of timelike surfaces in the Einstein universe $\EE$, the conformal
compactification of Minkowski 3-space realized as the space of null rays
through the origin of $\R^{2,3}$. $\EE$ is endowed with the Lorentz metric $g_\mathcal{E}$ 
induced from $\R^{2,3}$.
It is diffeomorphic to $S^1\times S^2$ and with respect to the conformal class
$[g_\mathcal{E}]$ is an example of an oriented, time-oriented, 
conformally flat Lorentz manifold.
As such, it admits a transitive action by the identity
component $\A$ of the pseudo-orthogonal group of $\R^{2,3}$, which is isomorphic
to the group of conformal transformations of $(\EE,[g_\mathcal{E}])$ preserving orientation 
and time-orientation.
%
In the paper we will also use the realization of $\EE$ as the symmetric $R$-space \cite{BDPP,GK,Ti} 
defined by the conjugacy class of parabolic subalgebras 
of $\aa$, the Lie algebra of $\A$. In dimension three, another model for the Einstein universe is  
the Grassmannian of oriented Lagrangian planes of $\R^4$ \cite{BCDG,MNSIGMA,THE}.
If we write $\MM=\mathbb{V}^2_-\oplus \mathbb{V}^3_+$ as the direct sum of a negative-definite 
2-dimensional 
subspace $\mathbb{V}^2_-$ and a positive-definite 3-dimensional subspace $\mathbb{V}^3_+$ and denote by ${S}^1_-$ 
and $S^2_-$, respectively, the unit spheres 
of $\mathbb{V}^2_-$ and $\mathbb{V}^3_+$, then the orientable, time-orientable submanifold 
${S}^1_-\times {S}^2_+$, with the Lorentzian structure inherited from $\MM$, provides another model 
(the \textit{pseudo-metric model}) of the Einstein universe. 

As in the case of definite signature, the 3-dimensional Lorentzian space forms, namely Minkowski, 
de Sitter and anti-de Sitter 3-spaces, can be conformally embedded as open domains of $\EE$, 
the so-called Minkowski, de Sitter, and anti-de Sitter chambers \cite{BCDG,DMN2016,EMN}.

\vskip0.2cm
\noindent \textbf{Background and motivations.}
Let $f: X \to \mathbb{R}^{1,2}$ be a timelike immersion of a surface $X$ in Minkowski 3-space. At umbilical points, the second 
fundamental form $\mathrm{II}$ of $f$ is proportional to the first fundamental form $\mathrm{I}$ and the mean and Gauss 
curvatures $H$ and $K$ satisfy $H^2 - K=0$. For timelike surfaces in Minkowski space, $H^2 - K$ can vanish also 
at nonumbilical points \cite{KM1983}, in which case $\mathrm{I}$ and $\mathrm{II}$ must share a null direction and the shape operator is
non-diagonalizable over $\mathbb{C}$. 
Following Clelland \cite{Cl1}, a timelike immersion $f$ 
is called (totally) \textit{quasi-umbilical} if at every point of $X$ its shape operator $A$ is non-diagonalizable 
over $\mathbb{C}$, 
or equivalently, if at every point of $X$ the trace-free part $A^0$ of the shape operator  
satisfies $\det ({{A}^0})=0$, with ${{A}^0}\neq 0$.

\vskip0.1cm
\noindent\textit{Conformal invariance.}
Since the property of being quasi-umbilical only depends on the conformal class of the induced Lorentz 
metric, we shall profit from studying it in the more appropriate context of conformal Lorentz geometry
by considering quasi-umbilical timelike immersions in the Einstein universe $\EE$.
Beside quasi-umbilical surfaces, there are three other classes
of timelike immersions of constant type,\footnote{Cf. \cite[p. 9]{JeLN} for the notion
of immersion of constant type in a homogeneous space.}
namely the class of elliptic surfaces, for which $\det ({{A}^0})>0$ on $X$, the class of hyperbolic surfaces, 
for which $\det ({{A}^0})<0$ on $X$, and the class
of surfaces for which ${A}^0$ vanishes identically on $X$. The latter consists of totally umbilical embedded Lorentzian tori. 

\vskip0.1cm
\noindent\textit{Related work.}
Clelland \cite{Cl1} classified quasi-umbilical surfaces in Min\-kow\-ski 3-space as ruled surfaces whose 
rulings are all null lines, with the additional property
that any null curve $\gamma$ on the surface transversal to the rulings is nondegenerate (i.e., the
vectors $\gamma'$, $\gamma''$ are linearly independent at each point).
At about the same time, The \cite{THE} investigated 
the conformal Lorentz geometry of
timelike surfaces, including quasi-umbilical ones, 
which he called 2-\textit{parabolic}, 
in connection with his study of the geometry of second-order Monge-Ampere equations.
More recently, Burstall, Musso, and Pember \cite{BMP-ASN} have shown that quasi-umbilical 
surfaces arise as one of the 
four classes of orthogonal surfaces of codimension 2 harmonic sphere congruences,
which also include $S$-Willmore surfaces in the sense of Ejiri \cite{Ejiri}, constant mean curvature 
surfaces in 3-dimensional space forms, 
and surfaces of constant lightcone mean curvature in 3-dimensional lightcones.
There has recently been also a considerable interest in
elliptic and hyperbolic timelike surfaces, especially in connection with the study of the Lorentz 
counterparts of Willmore and isothermic immersions 
\cite{BeSchae, Cal, DEWA, DM1, Ma1, Ma2, NieMaWang, Palmer, THE, Wa}.
Finally, we observe that conformal Lorentz geometry is related to Lie sphere geometry, 
a classical topic in differential geometry which has recently received much attention 
especially in relation with the study of Dupin submanifolds \cite{CC,Ce,Cec-Jen-98}, 
Lie applicable surfaces \cite{MNTo, Pe}, and the theory of integrable systems 
\cite{BH-J, Ferapontov}. Actually, the Einstein universe $\EE$ can be viewed as a double 
covering of the Lie quadric $\mathcal{Q}^3$ in Lie sphere geometry, whose elements 
correspond to oriented circles and points in $\mathbb{R}^2$ \cite{BlaschkeIII, JMNbook}. 

\vskip0.1cm
\noindent\textit{Purposes.}
A first purpose of this paper is to underline and clarify the links of quasi-umbilical immersions
with the general theory of isothermic submanifolds in symmetric $R$-spaces (in the sense of
\cite{BDPP}) and with the classical deformation problem of submanifolds in homogeneous spaces 
\cite{Ca2, BHJPR, JeLN, JeDg,Mu,MNTo,Pe}.  
A second purpose is to give a detailed geometric description of quasi-umbilical surfaces 
by showing that all quasi-umbilical surfaces can be constructed from suitable null curves in
the neutral four-space form 
$$
\SS=\{V\in \MM \mid \langle V,V\rangle =1\} \cong 
 \mathrm{O}^{\uparrow}_+(2,3)/ \mathrm{O}^{\uparrow}_+(2,2).
   $$

\vskip0.2cm
\noindent\textbf{Description of results.} Adapting to our case a basic construction in M\"obius geometry 
that goes back to the work of Blaschke and Thomsen  \cite{BlaschkeIII,Bryant-JDG1984, HJbook,Thomsen}, 
we can define the conformal Gauss map $\NN :X\to \SS$ of a timelike immersion. 
(For the study of spacelike immersions in conformal Lorentz geometry, and more specifically of spacelike Willmore surfaces,
we refer to \cite{AP} and \cite{IM}.)
Geometrically, $\NN$ associates to a point ${x}\in X$ the unique totally umbilical timelike torus 
of the Einstein universe having 
second order analytic contact with $f$ at $f({x})$,
the so-called \textit{central torus} of the surface at $f(x)$.  If $f$ is quasi-umbilical, then
$\NN$ has rank 1, so that $f$ is the envelope of 
a 1-parameter family of totally umbilical Lorentz tori. The second envelope of the family, denoted by $f^{\sharp}:X\to \EE$, 
is referred to as the \textit{dual} of $f$.  
A quasi-umbilical immersion $f$ is called \textit{regular} if $\Lambda$,
the leaf space of the vertical distribution $\ker d\NN$, is a connected 1-dimensional manifold, 
so that $\NN$ factors on $\Lambda$ 
through a submersion 
$\pi_{\Lambda}:X\to \Lambda$ with connected fibers
and a null curve $\Gamma :\Lambda\to \SS$, that is, $\NN=\Gamma\circ \pi_{\Lambda}$. 
The curve $\Gamma$ is called the
\textit{directrix curve} of $f$. Locally, every quasi-umbilical immersion is regular.
The first main result of the paper is the following.

\vskip0.2cm

\noindent {\bf Theorem A.}\,\,\textit{Let $f : X \to \EE$ be a quasi-umbilical timelike immersion. Then the conformal Gauss map 
$\NN:X\to \SS$ of $f$ is harmonic.
In addition, if $f$ is regular, then
$f$ is isothermic and its infinitesimal deformations depend on one arbitrary function in one variable}.

\vskip0.2cm
According to whether the quadratic form $\langle d\NN,df^{\sharp} \rangle$ is identically zero or not,
a quasi-umbilical immersion $f$ is called \textit{exceptional}, respectively, \textit{general}.
We will separately describe the geometric structure of the two classes of immersions under the
assumption of regularity.

\vskip0.2cm
\noindent\textit{Exceptional quasi-umbilical surfaces.}
First, we prove that a quasi-umbilical $f$ is exceptional if and only if its directrix curve 
$\Gamma$ is a \textit{biisotropic} curve of $\SS$, meaning that $\Gamma'$ 
and $\Gamma''$ are linearly independent null vector fields (cf. Proposition \ref{prop:except-biisotropic}).  
We show that biisotropic curves can be constructed from suitable plane curves of a unimodular affine plane.
Then, associated to any biisotropic curve $\Gamma : \Lambda\to \SS$, we consider the \textit{normal tube} 
 ${\mathbb T}_\Gamma\cong \Lambda\times S^1$ of $\Gamma$ and
construct a canonical map $f_{\Gamma}:{\mathbb T}_\Gamma\to \EE$, the \textit{tautological map} 
of $\Gamma$ (cf. Section \ref{s3.1} and Definition \ref{def:taut-map}).
The second main result accounts for the geometric structure of exceptional quasi-umbilical surfaces.

\vskip0.2cm

\noindent {\bf Theorem B.}\,\,\textit{Let $\Gamma :\Lambda\to \SS$ be a biisotropic curve. 
The tautological map $f_\Gamma$ of $\Gamma$ is a timelike immersion with conformal Gauss map $\Gamma\circ \pi_\Lambda$.
The umbilic locus of $f_{\Gamma}$ is the disjoint union of two {null curves} 
${\rm C}_{\pm}\subset {\mathbb T}_\Gamma$ 
that disconnect  ${\mathbb T}_\Gamma$ into two open sets,
${\mathbb T}^{\pm}_\Gamma$. Moreover, the restrictions of $f_{\Gamma}$ to ${\mathbb T}^{\pm}_\Gamma$ are  
exceptional quasi-umbilical immersions.
Conversely, if $f:X\to \EE$ is a regular quasi-umbilical immersion of exceptional type 
and $\Gamma:\Lambda\to \SS$ is its directrix curve, then $f(X)$ is contained in either
$f_{\Gamma}({\mathbb T}^{+}_\Gamma)$ or $ f_{\Gamma}({\mathbb T}^{-}_\Gamma)$.}

\vskip0.2cm
{This implies that, locally, exceptional quasi-umbilical surfaces arise as immersions of
a component of the normal tube of a biisotropic curve of $S^{2,2}$.}
Another consequence is that quasi-umbilical immersions of exceptional type depend on one function 
in one variable, namely the affine curvature of the affine curve generating the directrix curve. 
Moreover, they are second order 
deformations of each other (cf. Proposition \ref{pr1.s3.2}). 
The general construction will be illustrated for the case in which the generating affine curve is a conic.

\vskip0.2cm
\noindent\textit{General quasi-umbilical surfaces.} 
We start by showing that the directrix curve $\Gamma$ of a general quasi-umbilical immersion 
is a \textit{generic} null curve, meaning that $\Gamma''$ can be either 
timelike or spacelike. 
Then, proceeding in analogy with \cite{MNCQG,MNSIAM,MN-NONLIN2010}, for a generic $\Gamma$ we introduce
a preferred parametrization (proper time) and construct
a canonical moving frame along the curve, from which one determines
the two fundamental differential invariants of $\Gamma$, namely the \textit{left and right curvatures} 
$\kappa_\lambda$ and $\kappa_\varrho$.
The canonical moving frame is used to build two normal tubes ${\mathbb T}_{\lambda}$
and ${\mathbb T}_{\varrho}$ along 
 $\Gamma$, the \textit{left and right normal tubes} of $\Gamma$, 
naturally identified by a diffeomorphism $J^\varrho_\lambda$. The left and right tubes intersect along 
two null curves ${\rm C}_{\pm}$, whose complementary set 
has two connected components,
${\mathbb T}_\lambda^{\pm}$ and ${\mathbb T}_{\varrho}^{\pm}$, the \textit{parabolic components} 
of the normal tubes.
The canonical projections $\pi_\lambda :{\mathbb T}_\lambda \to \Lambda$ and
$\pi_\varrho :{\mathbb T}_\varrho \to \Lambda$ make ${\mathbb T}_\lambda$
and ${\mathbb T}_\varrho$ into circle bundles over $\Lambda$. We then construct
two canonical maps, $f_{\lambda}:{\mathbb T}_\lambda\to \EE$ and 
$f_{\varrho}:{\mathbb T}_\varrho\to \EE$,
the \textit{left and right tautological maps} of $\Gamma$.
The third main result accounts for the structure of general quasi-umbilical surfaces.

\vskip0.2cm

\noindent {\bf Theorem C.}\,\,\textit{Let $\Gamma :\Lambda\to \SS$ be a generic null curve. 
\vskip0.02cm
 $(1)$ The tautological map $f_{\lambda}$ is a timelike immersion with 
umbilic locus ${\rm C}_{+}\cup {\rm C}_{-}$ and conformal Gauss map ${\mathcal N}_{f_\lambda}=\Gamma\circ \pi_{\lambda}$. 
The restrictions $f^{\pm}_{\lambda}$ of $f_{\lambda}$ to the parabolic components 
${\mathbb T}_\lambda^{\pm}$ of the left normal tube ${\mathbb T}_\lambda$
are quasi-umbilical immersions of general type 
with dual maps $(f^{\pm}_{\lambda})^{\sharp}=f^{\pm}_{\varrho}\circ J^{\varrho}_{\lambda}$. 
\vskip0.02cm
$(2)$ The tautological map $f_{\varrho}$ is a timelike immersion with 
umbilic locus ${\rm C}_{+}\cup {\rm C}_{-}$ and conformal Gauss map ${\mathcal N}_{f_\varrho}=\Gamma\circ \pi_{\varrho}$. 
The restrictions $f^{\pm}_{\varrho}$ of $f_{\varrho}$ to the parabolic components ${\mathbb T}_\varrho^{\pm}$ 
of the right normal tube ${\mathbb T}_\varrho$
are quasi-umbilical immersions of general type with
dual maps $(f^{\pm}_{\varrho})^{\sharp}=f^{\pm}_{\lambda}\circ J_{\varrho}^{\lambda}$. 
\vskip0.02cm
Conversely, let $f:X\to \EE$ be a regular quasi-umbilical 
immersion of general type. Then $f(X)$ is contained in either 
$f_{\lambda}({\mathbb T}_{\lambda}^{+})$ or $f_{\lambda}({\mathbb T}_{\lambda}^{-})$, 
where $f_{\lambda}$ is the left 
tautological immersion originated by the directrix curve $\Gamma$ of $f$.
}

\vskip0.2cm

Theorem \ref{thmC} implies that quasi-umbilical immersions of general type depend on two
functions in one variable, namely the left and right curvatures of their directrix curves.
As for the deformation problem, we show that two quasi-umbilical immersions 
of general type $f$ and $\tilde{f}$ are second order conformal deformations of each other if and only 
if they have the same helicity (to be specified) and if either $\kappa_\lambda = \tilde{\kappa}_\lambda$
or $\kappa_\varrho = \tilde{\kappa}_\varrho$ (cf. Remark \ref{r:def-gen-type}).
{The construction will be illustrated by an example at the end of Section \ref{s4}}.

It is interesting to note that the construction of the canonical timelike immersions arising from biisotropic or 
generic null curves of $\SS$ is algorithmic and only involves algebraic manipulations and derivatives. 
The parametrizations of quasi-umbilical immersions given in \cite{Cl1} can be obtained restricting the maps 
$f_{\Gamma}$
to the subdomains of $\Lambda\times S^1$ that are sent into a Minkowski chamber of $\EE$ \cite{DMN2016,EMN}. 
Moreover, our construction allows us to exhibit quasi-umbilical immersed and embedded tori or cylinders, 
a phenomenon that can never occur in 3-dimensional Minkowski geometry.

\vskip0.2cm

The material is organized as follows. Section \ref{s1} briefly reviews some basic facts about the 
geometry of the Einstein universe 
(cf. \cite{BCDG,DMN2016,EMN,GK, GS, MNSIGMA,THE}) and
describes the AdS chambers and the ``toroidal model" of the Einstein universe which, in turn, is used to 
visualize 
the geometric features of quasi-umbilical surfaces. 
Section \ref{s2} recalls the construction of the conformal Gauss map and of the 
dual map of a quasi-umbilical immersion. 
We then recall the notion of an isothermic surface, adapting to 
the context of timelike surfaces in $\EE$
the terminology of the general theory of isothermic submanifolds in
symmetric $R$-spaces, and discuss the related concept of second order deformation.
Section \ref{s2} concludes with the proof of Theorem \ref{thmA}. 
Section \ref{s3} describes the geometric structure of exceptional quasi-umbilical immersions 
and proves Theorem \ref{thmB}. 
Section \ref{s4} accounts for the geometric structure of general quasi-umbilical immersions and proves Theorem \ref{thmC}. 

\section{Preliminaries}\label{s1}

\subsection{The automorphism group}

Consider $\R^5$ with the  scalar product of signature $(2,3)$ 
\begin{equation}\label{scalar-pro}
  \langle X,Y\rangle = -(X^0Y^4+X^4Y^0)-X^1Y^1+X^2Y^2+X^3Y^3= \T X\, h\,Y,
   \end{equation}
where $h=(h_{ij})$ and $h_{ij}=h_{ji}$. Choose an oriented spacelike
3-dimensional vector subspace $\mathbb{V}^3_+\subset \R^5$ and a positive-oriented orthogonal
basis $V_1$, $V_2$, $V_3$ of $\mathbb{V}^3_+$. 

Let $\mathtt{D}_2$ be the set of all nonzero decomposable bivectors $X\wedge Y\in \Lambda^2(\R^5)$ such that
the restriction of \eqref{scalar-pro} to $\mathrm{span}\{X,Y\}=: [X\wedge Y]$ is negative definite.
An element $X\wedge Y\in \mathtt{D}_2$ is said to be 
{\textit{positive}}
if ${\det}(X,Y,V_1,V_2,V_3)>0$. The set of all such positive bivectors is a connected component of $\mathtt{D}_2$, 
denoted by ${\mathtt D}^{\uparrow}_2$. 
The boundary $\partial {\mathtt D}^{\uparrow}_2$ of ${\mathtt D}^{\uparrow}_2$ consists of all 
isotropic bivectors $V\wedge W$ (i.e., 
$V\wedge W\neq 0$,  
$\langle V,V\rangle=   \langle W,W\rangle=   \langle V,W\rangle=0$) satisfying the 
condition
{$\langle X\wedge Y , V \wedge W \rangle =\langle X,V\rangle\langle Y,W\rangle - \langle Y,V\rangle\langle X,W\rangle > 0$}, 
for every $X\wedge Y\in {\mathtt D}^{\uparrow}_2$. 
The elements of $\partial {\mathtt D}^{\uparrow}_2$ 
are said to be \textit{future-directed}. 

\vskip0.1cm
Let $\MM$ denote $\R^5$
with the scalar product \eqref{scalar-pro}, the orientation induced by the determinant, and the
time-orientation determined by ${\mathtt D}^{\uparrow}_2$.
The \textit{automorphism group} of $\MM$ is
the 10-dimensional Lie group $\mathrm{Aut}(\mathbb{R}^{2,3})$ of linear isometries of $\MM$ preserving the given orientation and
time-orientation.

Let $M_{\mathfrak{B}}(\Phi)$ be the matrix representing $\Phi\in \mathrm{Aut}(\mathbb{R}^{2,3})$ with respect to 
a chosen basis ${\mathfrak B} =({B}_0,\dots, {B}_4)$ of $\MM$. 
Similarly, let $G_{\mathfrak{B}}$ be the symmetric matrix representing the scalar product $\langle \cdot\,,\cdot \rangle$ 
with respect to ${\mathfrak B}$.  
For every choice of ${\mathfrak B}$,
the map 
{$\mathrm{Aut}\left({\mathbb{R}^{2,3}}\right) \ni \Phi\mapsto M_{\mathfrak{B}}(\Phi)\in \mathrm{GL}(5,\mathbb{R})$} 
is a faithful matrix representation of 
$\mathrm{Aut}\left(\mathbb{R}^{2,3}\right)$.
We say that:
\begin{itemize}

\item ${\mathfrak B}$ is a \textit{lightcone basis} 
if it is positive-oriented, if\footnote{Here $E^a_b$, $0\leq a,b\leq 4$, denotes the elementary matrix with 1
in the $(a,b)$ place and $0$ elsewhere.}
$$
  G_{{\mathfrak B}}= h = -(E^0_4+E^4_0)-E^1_1+E^2_2+E^3_3,
       $$
 and if the isotropic bivector ${B}_0\wedge ({B}_1+{B}_2)$
is future-directed;

\item ${\mathfrak B}$ is 
a \textit{pseudo-orthogonal basis} if it is positive-oriented, if
$$
   G_{{\mathfrak B}}=-E^0_0-E^1_1+E^2_2+E^3_3+E^4_4,
     $$
and if the isotropic bivector $({B}_0-{B}_4)\wedge({B}_1+{B}_2)$ 
is future-directed.
\end{itemize}
Let
\begin{equation}\label{intertwining1}
T_\mathrm{lp}=\frac{1}{\sqrt{2}}(E^0_0+E^4_0)+E^1_1+E^2_2+E^3_3+\frac{1}{\sqrt{2}}(E^4_4-E^0_4).
\end{equation}
Then, ${\mathfrak B}$ is a lightcone basis if and only if ${\mathfrak B}\cdot T_\mathrm{lp}$ is a pseudo-orthogonal basis. In particular, 
if $\mathfrak{E}= (E_0,\dots,E_4)$ is the standard basis of $\MM$,
$\mathfrak{P}= (P_0,\dots,P_4)= \mathfrak{E}\cdot T_\mathrm{lp}$ is called
the {\it standard pseudo-orthogonal basis} of $\MM$.

\vskip0.1cm
Choosing the standard basis $\mathfrak{E}$, 
the automorphism group $\mathrm{Aut}(\mathbb{R}^{2,3})$ is identified
with 
the identity component of the pseudo-orthogonal group of \eqref{scalar-pro}, namely
$$
  \A= \left\{F\in \SL(5,\R) \mid \T F\, h \, F = h,\, {F_0\wedge (F_1 +F_2) \in \partial {\mathtt D}^{\uparrow}_2} \right\},
     $$
where $F_0, \dots, F_4 \in \R^{2,3}$ denote the column vectors of $F$. If $F\in \A$, the ordered set
$(F_0, \dots, F_4)$ defines a lightcone basis of $\R^{2,3}$. Conversely, any lightcone basis arises in this way.
Thus the set of all lightcone basis of $\MM$ can be identified with $\A$.

\vskip0.1cm
The Lie algebra of $\A$  is the vector space
\[
 \aa=\{\mathbf{a} \in {\mathfrak{sl}(5,\mathbb R}) \mid  \T \mathbf{a}\,h + h\,\mathbf{a}=0\},
  \]
 with the commutator of matrices as a bracket. 
Taking
   \[
 \begin{array}{lllll}
 M^0_0=E^0_0-E^4_4,\! & M^1_0=E^1_0-E^4_1,\! & M^2_0=E^2_0+E^4_2,\! &  M^3_0=E^3_0+E^4_3,\!\\
 M^2_1=E^2_1+E^1_2,\! & M^3_1=E^3_1+E^1_3,\! &M^3_2=E^3_2-E^2_3\!,  &M^1_4=E^1_4-E^0_1,\!\\
 M^2_4=E^2_4+E^0_2,\! & M^3_4=E^3_4+E^0_3\\
   \end{array}
   \]
as a basis of $\aa$, the \textit{Maurer--Cartan form} $\varphi =(\varphi^i_j)= F^{-1}dF$ of $\A$ can be written as
\[
{\begin{split} 
\varphi  &= \varphi^0_0M^0_0+ \varphi^1_0M^1_0+ \varphi^2_0M^2_0+ \varphi^3_0M^3_0+ \varphi^2_1M^2_1\\
& \qquad 
 + \varphi^3_1M^3_1+ \varphi^3_2M^3_2+
    \varphi^1_4M^1_4+ \varphi^2_4M^2_4+ \varphi^3_4M^3_4,
   \end{split}}
   \]
where the left-invariant 1-forms {$\varphi^0_0=-\varphi^4_4$, $\varphi^1_0=-\varphi^4_1$, $\varphi^2_0=\varphi^4_2$, 
$\varphi^3_0=\varphi^4_3$, $\varphi^2_1= \varphi^1_2$, $\varphi^3_1=\varphi^1_3$, 
$\varphi^3_2=-\varphi^2_3$, $\varphi^1_4=-\varphi^0_1$, $\varphi^2_4= \varphi^0_2$, $\varphi^3_4=\varphi^0_3$}
are linearly independent and span the dual Lie algebra $\aa^*$.
The Maurer--Cartan form  satisfies the \textit{Maurer--Cartan equations} 
\begin{equation}\label{MCeq}
 d \varphi = - \varphi\wedge  \varphi.
   \end{equation}

\subsection{The $(1+2)$-Einstein universe} 

Consider the orthogonal direct sum decomposition $\MM=\mathbb{V}^2_-\oplus \mathbb{V}_+^3$
into the oriented, negative definite subspace
$\mathbb{V}_-^2 =[P_0\wedge P_1]$ and the spacelike subspace
$\mathbb{V}_+^3=[P_2\wedge P_3\wedge P_4]$, where $\mathfrak{P}=(P_0,\dots, P_4)$ is the standard
pseudo-orthogonal basis of $\MM$. 
Let $(x_0,\dots,x_4)$ be the Cartesian coordinates with respect to $\mathfrak{P}$.
Denote by ${S}^1_-$ the unit circle 
of $\mathbb{V}_-^2$, by ${S}^2_+$ 
the unit sphere of $\mathbb{V}_+^3$, by ${J}$ the counterclockwise rotation of $\pi/2$ in 
the oriented plane $\mathbb{V}_-^2$, and by $V_-$ and $V_+$ the orthogonal projections of 
$V\in \MM$ onto $\mathbb{V}_-^2$ and $\mathbb{V}_+^3$,  respectively.  
Let $\mathcal{E}^{1,2}$ be the Cartesian product  ${S}^1_- \times {S}^2_+$. 
The scalar product on $\MM$ induces a Lorentzian pseudo-metric $g_{{\mathcal E}}$ on $\mathcal{E}^{1,2}$. 
The normal bundle of $\mathcal{E}^{1,2}$ is spanned by the 
unit normal vector fields defined by $\mathbf{n}_1|_{V}=V_-$ 
and $\mathbf{n}_2|_{V}=V_+$, for each $V \in \EE$.  Contracting the volume form $dx^0\wedge \dots \wedge dx^4$ 
with $\mathbf{n}_1$ and $\mathbf{n}_2$
yields a volume form $d{V}_{\mathcal{E}}$ on $\mathcal{E}^{1,2}$. We time-orient $\mathcal{E}^{1,2}$ by 
requiring that the timelike vector field $\mathbf{t}|_{V}={J}V_-$  is future-oriented.

\begin{defn}
The Lorentzian manifold $(\mathcal{E}^{1,2},g_{\mathcal{E}})$, with the above specified orientation
and time-orientation, is called the {\it $(1+2)$-Einstein universe}.
\end{defn}

For each nonzero vector $X\in \MM$, let $[X]$ denote the oriented line spanned by $X$ (i.e., the ray of $X$).
The map $\mathcal{E}^{1,2} \ni X \mapsto [X]$ allows us to identify $\mathcal{E}^{1,2}$ with the manifold of oriented 
null rays of $\MM$. 
We will make no distinction between these two models. 

\begin{remark}
As a model of the Einstein universe it is often used the non-orientable manifold 
$\widehat{\mathcal{E}}^{1,2}$
of null (unoriented) lines of $\MM$. This is the quotient of $ \mathcal{E}^{1,2}$ by the orientation-reversing 
involutory isometry $[X]\mapsto [-X]$. 
\end{remark}

Under the identification above, $\A$ acts effectively and transitively on $\mathcal{E}^{1,2}$ 
on the left by $F\cdot [X ] =[F\cdot X]$, for each $F\in \A$ and $[X]\in \mathcal{E}^{1,2}$. 
This action preserves the oriented and time-oriented conformal Lorentzian structure of $\mathcal{E}^{1,2}$.
It is a classical result that every restricted conformal transformation of $\EE$ is induced by a unique element of 
$\A$ (cf. \cite{BCDG,EMN}). Accordingly, we call $\A$ 
the (\textit{restricted}) \textit{conformal group} of $\EE$.

\subsection{The AdS chambers and the neutral space form $\SS$} 

Let $S\in \MM$ be a unit spacelike vector. The \textit{positive and negative AdS (anti-de Sitter) chambers} 
of $\EE$ determined by ${S}$ 
are the open sets
\[
 \mathcal{A}^{\varepsilon}_{{S}}=\{[{V}]\in \EE \mid 
 \sgn\langle {V},{S}\rangle
 = \varepsilon \}\subset \EE,\quad \varepsilon =\pm1. 
 \]
 The common boundary of the AdS chambers is the 2-dimensional torus 
 $\partial  \mathcal{A}_{{S}} = \{[{V}]\in \EE \mid \langle {V},{S}\rangle = 0 \}\cong S^1 \times S^1$, called the \textit{AdS wall}.  
Accordingly, $\EE$ splits as the disjoint union of the two AdS chambers and the AdS wall.  
 The AdS walls are totally umbilical timelike surfaces (Lorentz tori) of $\EE$. Contracting $d{V}_{\mathcal{E}}$ 
 with ${S}$ defines a volume form on $\partial  \mathcal{A}_{{S}}$.  
 Thus, the unit spacelike vectors of $\MM$ are in 1:1 correspondence with the oriented AdS walls. 
 In this way, the smooth hyperquadric 
 $$
   \SS=\{{S}\in \MM \mid \langle {S}, {S}\rangle = 1\}\subset \MM
     $$ 
can be viewed as the manifold of all oriented AdS walls of $\EE$.
The scalar product of $\MM$ induces a neutral $(2,2)$ pseudo-metric 
$g_{{S}}$ on $\SS$ \cite{JeRi}. The group $\A$ acts transitively on $\SS$ on the left 
preserving $g_{{S}}$. 
With respect to such an action, $\SS$ is a symmetric space. The map
 \begin{equation}\label{S22}
 \mathbb{V}_-^2\oplus {S}^2_+\ni V_-+V_+  \longmapsto V_-+\sqrt{1-\langle V_-,V_-\rangle}V_+\in \SS 
     \end{equation}
 is a smooth diffeomorphism. So, as a manifold, $\SS$ is the Cartesian product of $\R^2$ with a 2-dimensional sphere.

 \begin{remark}
As a model for the \textit{anti-de Sitter 3-space}, we take the hyperquadric $ \mathcal{A}\subset \R^4$ consisting of 
all vectors $\mathbf{p}=({x}, {y})$, ${x} =(x^1,x^2)$, ${y}=(y^1,y^2)$ of $\R^4$, such that
 $-(x^1)^2-(x^2)^2+(y^1)^2+(y^2)^2=-1$, equipped with the Lorentzian structure induced by 
 the neutral scalar product
$ (\mathbf{p},\tilde{\mathbf{p}})_{(2,2)}=-x^1 \tilde{x}^1-x^2 \tilde{x}^2+y^1 \tilde{y}^1+y^2 \tilde{y}^2$.  
Let ${\mathfrak{B}}$ be a pseudo-orthogonal basis such that ${B}_4= {S}$. 
Then, $\mathcal{A}^{\varepsilon}_{{S}}$ is the 
image of $\mathcal{A}$ via the conformal embedding
\[
  \mathcal{A} \ni ({x},{y}) \mapsto
   \left[x^1{B}_0+x^2{B}_1+y^1{B}_2+y^2{B}_3+\varepsilon {B}_4\right] \in \mathcal{A}^{\varepsilon}_{{S}}. 
   \]
Let  $\mathrm{T}\subset \R^3$ be the open solid torus 
swept out by rotating around the $z$-axis the (open) unit disk 
in the $xz$-plane centered at $(2,0,0)$.  
Using local coordinates $0\leq r <1$, $\phi$, $\theta \in [0,2\pi)$, the points of $\mathrm{T}$ can be parametrized by
$$
P(r,\phi,\theta) = \Big((r\cos\phi+2)\cos\theta, (r\cos\phi+2)\sin\theta, r\sin\phi\Big).
$$
The map ${\mathcal T}_{\varepsilon}: \mathrm{T}\to \mathcal{A}^{\varepsilon}_{{S}}$,
defined by
$$
   P(r,\phi,\theta)\mapsto \left[r\cos(\phi)B_0+r\sin(\phi)B_1+\cos(\theta)B_2+\sin(\theta)B_3+\varepsilon \sqrt{1-r^2}B_4\right]
     $$
is  a smooth diffeomorphism which extends by continuity to the boundary $\partial \mathrm{T}$ of the solid torus 
$\mathrm{T}$, giving rise to a diffeomorphism 
$$
    \partial \mathrm{T}\ni P(1,\phi,\theta) \mapsto [\cos(\phi)B_0+\sin(\phi)B_1+\cos(\theta)B_2+\sin(\theta)B_3 ]\in 
      \partial  \mathcal{A}_{{S}}
        $$ 
between the boundary $\partial \mathrm{T}$ of the solid torus and the AdS wall $\partial  \mathcal{A}_{{S}}$.
Thus, topologically, $\EE$ can be identified with the disjoint union
$\overline{\mathrm{T}}\times \{-1,1\}$ of two copies of the closed solid torus
modulo the equivalence relation $[(P,\epsilon)]_{\sim}=\{(P,\epsilon)\}$, 
if $P\in \mathrm{Int}(\mathrm{T})$, and $[(P,\epsilon)]_{\sim}=\{(P,1)(P,-1)\}$, if $P\in \partial \mathrm{T}$.
In the following, we will use the ``toroidal" projections ${\mathcal T}_{\varepsilon}^{-1} : \overline{ \mathcal{A}}^{\varepsilon}_{{S}}\to \overline{\mathrm{T}}$ to visualize the geometric content of some of our results. 
\end{remark}

\subsection{Other models of the Einstein universe}\label{models-of-EU}

\subsubsection{$\EE$ as symmetric $R$-space}
The transitive action of $\A$ on $\EE$ defines a principal ${H}_0$-bundle
\begin{equation}\label{H0-pric}
 \pi_{{\mathcal E}}:  \A \ni F\longmapsto [{F}_0]\in \EE,
        \end{equation}
        where ${H}_0$ is the isotropy subgroup of $\A$ at $[{E}_0]$, namely
the 7-dimensional connected  Lie subgroup of $\A$ 
\[
H_0 = \left\{ R(r,L,x) =
\begin{pmatrix}
e^r & e^r\, \T x \epsilon_{1,2} L & \frac{e^{r}}{2}\, \T x \epsilon_{1,2} x	\\
0   &  L  &   x  \\
0   &   0  & e^{-r}
\end{pmatrix} \mid  \begin{array}{l} r \in \mathbb{R}, \, x \in \mathbb{R}^{3}\\
\T L = \epsilon_{1,2} L^{-1} \epsilon_{1,2}  \\ 
\epsilon_{1,2} = \mathrm{diag}(-1,1,1) \end{array} 
\right\}.
      \]
The Lie algebra of $H_0$ is the subalgebra of $\aa$
\[
\mathfrak{h}_0 = \left\{ \begin{pmatrix}  r& \T x \epsilon_{1,2} &  0\\
0 & A & x\\
0 & 0 & -r
\end{pmatrix} \mid \begin{array}{l} r \in \mathbb{R}, \, x \in \mathbb{R}^{3}\\
\T A =- \epsilon_{1,2} A \epsilon_{1,2} \\
\epsilon_{1,2} = \mathrm{diag}(-1,1,1) \end{array} 
\right\}.
\]
The polar space $\mathfrak{h}_0^\perp$ of $\mathfrak{h}_0$ with respect to the Killing form of $\aa$ is 
\[
\mathfrak{h}_0^\perp = \left\{ \begin{pmatrix}  0& \T x \epsilon_{1,2} &  0\\
0 & 0 & x\\
0 & 0 &0
\end{pmatrix} \mid \begin{array}{l}  x \in \mathbb{R}^{3}\\
\epsilon_{1,2} = \mathrm{diag}(-1,1,1) \end{array} 
\right\},
\]
which is an abelian subalgebra of $\aa$.
Thus $\mathfrak{h}_0$ is a \textit{height} 1 parabolic subalgebra of $\aa$ (cf. \cite{BDPP} for the 
definitions and more details).
More generally, for each $[V]\in \EE$, the polar of ${\mathfrak h}_{[V]}$, the Lie algebra of
the isotropy subgroup ${H}_{[V]}$ at $[V]$, is an abelian subalgebra of $\aa$, and hence 
${\mathfrak h}_{[V]}$ is a \textit{height} 1 parabolic subalgebra of $\aa$.

 The adjoint representation of  $\A$ induces an action on the Grassmannian of 7-dimensional linear subspaces of $\aa$ which, 
 in turn, gives rise to a transitive action of $\A$ on the orbit ${\mathcal O}$ through ${\mathfrak h}_0={\mathfrak h}_{[E_0]}$. 
 Accordingly, the orbit ${\mathcal O}$ acted upon transitively by $\A$ is a self-dual symmetric $R$-space
 (cf. \cite{BDPP, GK,Ti}).\footnote{A parabolic subalgebra $\mathfrak{q}\subset \aa$ is said to be \textit{complementary} to
 a given parabolic subalgebra $\mathfrak{p}$ if $\aa = \mathfrak{p} + \mathfrak{q}$. If $\mathfrak{q}$ is complementary 
to $\mathfrak{p}$, the orbit $\mathcal{O}^*$ of $\mathfrak{q}$ coincides with the set of all parabolic subalgebras complementary to
$\mathfrak{p}$. The orbit $\mathcal{O}^*$ is the \textit{dual} of $\mathcal{O}$. The symmetric $R$-space $\mathcal{O}$  
is \textit{self-dual} if $\mathcal{O}^*$ and $\mathcal{O}$ coincide.}   
 Since ${\mathfrak h}_{[F\cdot V]}=Ad_{F}\cdot {\mathfrak h}_{[V]}$ and ${\mathfrak h}_{[V]}={\mathfrak h}_{[W]}$ if and only 
 if $[V]=[W]$, the map $ \EE \ni[V]\mapsto {\mathfrak h}_{[V]}\in {\mathcal O}$ defines an equivariant smooth diffeomorphism. 
 This describes $\EE$ as the symmetric $R$-space given by 
 the conjugacy class $\mathcal{O}$ of height 1 parabolic
 subalgebras of $\aa$.

\vskip0.1cm
\subsubsection{$\EE$ as oriented Lagrangian Grassmannian}

The Einstein universe $\EE$ can also be realized as the \textit{Grassmannian of oriented Lagrangian planes} of $\R^4$.
This description of $\EE$ is specific of dimension three. 
Let $(\R^4, \omega)$ be $\R^4$ with the standard symplectic form $\omega(x,y)=\T xJy$, where
$J = \big(\begin{smallmatrix} 0&I_2\\-I_2&0\end{smallmatrix}\big)$. Next, on the 6-dimensional vector space
${\Lambda^2}\left({(\R^4)}^*\right)$ 
consider the neutral scalar product $(\alpha, \beta)$ of signature $(3,3)$, defined by $(\alpha,\beta)\omega^2 = \alpha\wedge\beta$.
By construction, $\omega$ is a spacelike vector, i.e., $(\omega, \omega)=1$, which implies that the restriction of $(\cdot\,,\cdot)$
to the 5-dimensional polar space $[\omega]^\perp$ has signature $(2,3)$.
Hence, $[\omega]^{\perp}$ together with $(\cdot\,,\cdot)$ can be identified with $\MM$.  
Thus, the Einstein universe $\EE$
(i.e., the manifold of null rays of $\R^{2,3} \cong [\omega]^\perp$) can be
canonically identified with the Grassmannian of oriented Lagrangian planes of $(\R^4,\omega)$ (cf. \cite{DMN2016, THE}).

\section{Quasi-umbilical timelike immersions}\label{s2}

\subsection{Timelike immersions}\label{s2.1}

Let $X$ be an \textit{oriented} connected surface. An immersion $f:X\to \EE$ is said to be \textit{timelike} if the 
induced metric $f^*(g_{{\mathcal E}})$ is Lorentzian. 
{Observe that $(X, f^*(g_{{\mathcal E}}))$ is time-orientable.}
When no confusion can arise, we use $f$ to denote both the map into $\EE$ and a lift of $f$
to $\MM$. 

Let $\pi_0 : {\mathcal F}_0(f)\to X$, where
 $$
   {\mathcal F}_0(f)=\{(x,F)\in X\times \A \mid  f(x)=\pi_0(F) = [F_0]\},
      $$
be the pullback by $f$ of the principal $H_0$-bundle $\pi_{\mathcal{E}}:\A\to \EE$.  
A \textit{conformal frame field} along $f$ is a local section of ${\mathcal F}_0(f)$, that is,
a smooth map $F:U\to \A$, where $U$ is an open set of $X$, such that $f = \pi_0 \circ F$.
%
For any such frame field we put $\phi =  F^*\varphi = (\phi^i_j)$. Then\footnote{We suppress the tensor 
symbol in the tensor products $\phi^j_k\otimes F_j$ occurring in the sum.}
\[
   dF_k =\sum_j \phi^j_kF_j \quad (k=0,\dots,4).
     \]
Given a conformal frame field on $U$, any other on $U$ is given by
\[
 \tilde{F} = F R,
   \]
where $R : U \to H_0$ is a smooth map. If $\tilde{\phi} = (\tilde{F})^\ast \varphi$, then
\begin{equation}\label{frame-change}
  \tilde{\phi} = R^{-1} \phi R + R^{-1} dR.
   \end{equation}

\begin{defn}
A conformal frame field $F:U \to \A$ is of \textit{first order} if at every point of $U$
\[
 \phi^3_0 = 0, \quad  \phi^1_0 \wedge \phi^2_0 >0.
   \]
    \end{defn}
It is easily seen that first order frame fields exist locally. Moreover, if $F$ is a first order frame field, 
then any other on $U$ is given by $\tilde{F} = F R$, for a smooth map $R : U \to H_1$, and
\[
  H_1 = \left\{R(r,s,x,y)  = \begin{pmatrix}
e^r & e^r\, \T x \epsilon_{1,1} B & e^r y & e^{r}    \frac{\T x \epsilon_{1,1} x +y^2}{2}\\
0    &  B   & 0  & x  \\
0    &  0   &1   & y \\
0    &   0  &0   & e^{-r}
\end{pmatrix} \in H_0
    \right\},
       \]
where 
\[
 r \in \mathbb{R},  \quad
B =\left(\begin{smallmatrix} \cosh{s} & \sinh{s}\\\sinh{s} & \cosh{s}\end{smallmatrix}\right), \, s\in \mathbb{R},
\quad x \in \mathbb{R}^{2}, \quad y\in \mathbb{R},
\quad\epsilon_{1,1} = \mathrm{diag}(-1,1).
  \]
The totality of first order frames along $f$ gives rise to a principal $H_1$-bundle $\pi_1 : {\mathcal F}_1(f)\to X$,
where ${\mathcal F}_1(f) = \{(x, F) \in X\times \A  \}$, being $F$ any local first order frame field along $f$.

\begin{figure}[ht]
\begin{center}
\includegraphics[height=6cm,width=6cm]{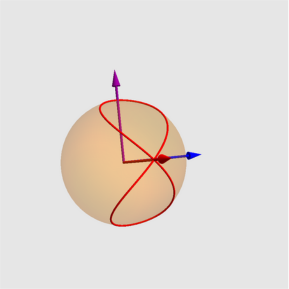}
\includegraphics[height=6cm,width=6cm]{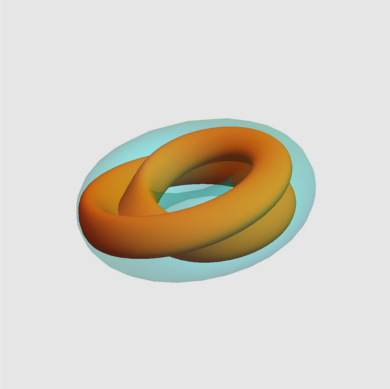}
\caption{\small The Viviani curve (a spherical lemniscate) (cf. \cite{Ha}) and the associated timelike immersion of the Klein bottle in the Einstein universe.}\label{FIG1}
 \end{center}
  \end{figure}

\begin{defn}
A first order frame field $F$ along $f$ is of \textit{second order} if
\begin{equation}\label{2.1}
  \phi^3_1\wedge \phi^2_0+\phi^3_2\wedge \phi^1_0=0.
   \end{equation}
\end{defn}

A proof of the existence of second order frame fields, using different notation, can be found in \cite[p. 99]{THE} 
(cf. also \cite{Bryant-JDG1984,JMNbook,Mu} for the M\"obius case).

\begin{prop}
Second order frame fields exist near any point of $X$. Moreover, if $F, \hat{F} : U\to \A$ are second 
order frame fields, then
$\hat{F} = F R$, where $R: U \to H_2$ is a smooth map from $U$ into the closed subgroup 
$$
     {H}_2=\{R(r,s,x,y) \in { H}_1 \mid  y = 0\}.
         $$
The second order frame fields along $f$ are the local sections of a reduced subbundle 
$\pi_2 : {\mathcal F}_2(f)\to X$ of $\pi_0 :{\mathcal F}_0(f)\to X$, with 
structure group ${H}_2$.
\end{prop}

\begin{defn}
Since the map $ {\mathcal F}_2(f)\ni (x,F) \mapsto F_3\in \SS$ is constant along the fibers of $\pi_2: {\mathcal F}_2(f)\to X$,
there exists a unique smooth map $\NN:X\to \SS$, called the \textit{conformal Gauss map} of $f$, such that $\NN\circ \pi_2=F_3$.
\end{defn}

\begin{remark}
The AdS wall $\partial  \mathcal{A}_{\NN|_{x}}$ is characterized by the property of having second order 
analytic contact with $f$ at $f(x)$ 
and is called the \textit{central torus} of the surface at $f(x)$ (see Figure \ref{FIG2}). The order of contact is strictly bigger 
than two if and only if $x$ is an \textit{umbilical point} of $f$ 
(i.e., if {$f(x)\wedge d\NN|_{x} = 0$}).
\end{remark}

\begin{figure}[ht]
\begin{center}
\includegraphics[height=6cm,width=6cm]{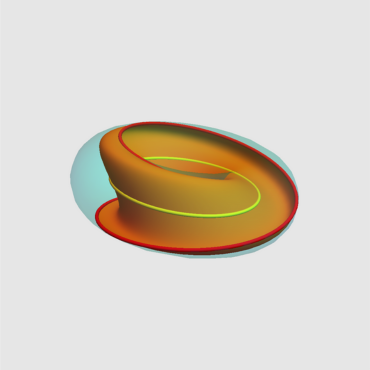}
\includegraphics[height=6cm,width=6cm]{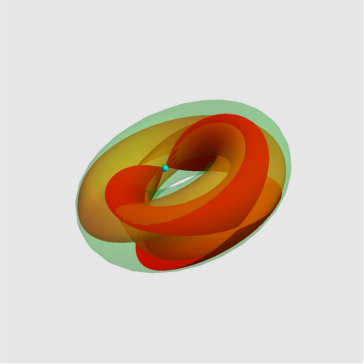}
\caption{\small On the left: the M\"obius strip bounded by the red curve, the intersection of the surface 
with the AdS wall. On the right: the portion of a central torus of the surface (at the point colored in cyan)
contained in the positive AdS chamber.}\label{FIG2}
\end{center}
\end{figure}

\begin{remark}\label{rr.s2.1}
On ${\mathcal F}_2(f)$, 
the 1-forms $\phi^3_0$ and $\phi^3_1\wedge \phi^2_0+\phi^3_2\wedge \phi^1_0$
are identically zero, and $\phi^1_0\wedge \phi^2_0 >0$.
Differentiating $\phi^3_0=0$, by \eqref{MCeq}, we compute
$\phi^3_1\wedge \phi^1_0+ \phi^3_2\wedge \phi^2_0=0$.  It follows, by Cartan's Lemma, that there exist 
smooth functions $h_{11}$, $h_{12}$, $h_{22}$ on ${\mathcal F}_2(f)$,
such that
\begin{equation}\label{2.2}
 \phi^3_1=h_{11}\phi^1_0+h_{12}\phi^2_0,\quad  \phi^3_2=h_{12}\phi^1_0+h_{22}\phi^2_0.
   \end{equation}
Because $\phi^3_1\wedge \phi^2_0+\phi^3_2\wedge \phi^1_0=0$, by \eqref{2.2},
it then follows that $h_{11} = h_{22}$.
Next, on ${\mathcal F}_2(f)$, we have $\phi^4_3=\phi^3_0 =0$, so we compute
\begin{equation}\label{2.2.bis}
 d{F}_3= \phi^0_3   {F}_0 + \phi^3_1  {F}_1 - \phi^3_2  {F}_2.
 \end{equation}
Because $h_{11} = h_{22}$ on ${\mathcal F}_2(f)$, it follows from \eqref{2.2.bis} that
\begin{equation}\label{2.2.tris}
  \langle d{F}_3,d{F}_3\rangle = (h_{11}^2-h_{12}^2)\left( -(\phi^1_0)^2 + (\phi^2_0)^2\right).
     \end{equation}
{Since $\mathcal{N}_f=F_3$ on ${\mathcal F}_2(f)$, we see that 
$\langle d\mathcal{N}_f, d\mathcal{N}_f\rangle$ is proportional to $ -(\phi^1_0)^2 + (\phi^2_0)^2$, 
which defines the induced
conformal structure on $X$.} This shows that the conformal Gauss map $\mathcal{N}_f: X \to \SS$ is \textit{weakly conformal}.  
\end{remark}

\begin{remark}
If $F: U \subset X \to \A$ is a first order frame field along $f$, the shape operator of $f$ with respect 
to $F$ can be defined at every point of $U$ by
the self-adjoint endomorphism of the tangent space given by
${A}_F= -\phi^3_1 F_1 + \phi^3_2 F_2$. The shape operator depends on the choice of $F$. Indeed,
 if $\tilde{F}$ and $F$ are first order frames on $U$ and $\tilde{F} = F R$, for $R = R(r,s,x,y) : U \to H_1$, it follows 
 from \eqref{frame-change} that ${A}_{\tilde{F}} = e^{-r} {A}_{{F}} - y \,\mathrm{Id}$. 
It is not difficult to verify that a first order frame field is of second order if $\tr \, {A}_{{F}}=0$. In particular, observe that 
the conformal factor in \eqref{2.2.tris} amounts to {the negative of} $\det\, {A}_{{F}}$, the determinant
of the trace-free part of the shape operator of $f$.
 \end{remark}

\begin{defn}
If $f:X\to \EE$ is a timelike immersion, then
$$\langle d\NN, d\NN\rangle  =\varrho_f f^*(g_{\mathcal E}),$$ 
where $\varrho_f$ is a smooth real-valued function. 
A point ${x}\in X$ is called
\begin{itemize}
\item  a 2-\textit{elliptic point} if $\varrho_f({ x})<0$;
\item  a 2-\textit{hyperbolic point} if $\varrho_f({ x})>0$;
\item a \textit{quasi-umbilical point} (or a 2-\textit{parabolic point}) if $\varrho_f({ x})=0$, with 
$f({x})\wedge d\NN|_{{ x}}\neq 0$;
\item an \textit{umbilical point} if 
$f({x})\wedge d\NN|_{{ x}}= 0$.  
If $x$ is umbilical,
$\varrho_f({ x})=0$. 
\end{itemize}
If all points of $X$ are of a fixed type, i.e., if $f$ is of \textit{constant type}, $f$ is 
said to be, respectively,  \textit{elliptic, hyperbolic, quasi-umbilical}, and \textit{totally umbilical}.
\end{defn} 

\begin{remark}
It is easily seen that the (trace-free part of the) shape operator of $f$ at ${x}$ is non-diagonalizable over $\mathbb{C}$ if and 
only if ${x}$ is quasi-umbilical.  
A timelike immersion $f$ is totally umbilical if and only if $f(X)$ is an open set of an AdS wall. 
With respect to a second order frame field $F$, the 2nd fundamental form of $f$ is locally given by 
$-\langle dF_0, d\mathcal{N}_f\rangle$.
\end{remark}

\begin{ex}
Let $({P}_0,\dots ,{P}_4)$ be the standard pseudo-orthogonal basis of 
$\MM$ and  $\gamma : \R\to \EE$ be the curve defined by
$$
     \gamma(s)={P}_0+\frac{1}{2}((1+\cos(2s)){P}_2+\sin(2s){P}_3+2\sin(s){P}_4).
         $$
Consider the 1-parameter group of isometries of $ \EE$ given by
{\[
   \begin{split}
     Q(t) &=\cos(2t)(E^0_0+E^1_1)-\sin(2t)(E^1_0-E^0_1)+E^2_2+\\
  &\quad +\cos(t)(E^3_3+E^4_4)+\sin(t)(E^4_3-E^3_4).
     \end{split}
      \]}
Then the map 
$$
 \hat{f}: {\mathbb T}^2=\R^2/2\pi \Z^2\ni [(s,t)] \mapsto Q(t)\cdot \gamma(s)\in \EE
  $$
is a timelike immersion trapped in the closure of the AdS chamber ${\mathcal A}_{{P}_3}$
(see Figure \ref{FIG2}). In addition, $\hat{f}$ is invariant under the involutory orientation-reversing diffeomorphism 
$\psi : (s,t)\in {\mathbb T}^2\to (-s,t+\pi)\in  {\mathbb T}^2$. Hence,  $\hat{f}$ descends to a generically 1:1 
immersion of the Klein bottle  ${\mathbb B}^2={\mathbb T}^2/\{\mathrm{Id},\psi\}$ (see Figure \ref{FIG1}).  
The immersion $\hat{f}$ has no umbilical points and there are two closed curves of quasi-umbilical points that 
disconnect ${\mathbb B}^2$ into two open sets consisting respectively of 2-elliptic and 2-hyperbolic points 
(see Figure \ref{FIG3}). {This construction works for every spherical lemniscate which is invariant under a
reflection with respect to a great circle of the 2-sphere.}
 \end{ex}

\begin{figure}[ht]
\begin{center}
\includegraphics[height=6cm,width=6cm]{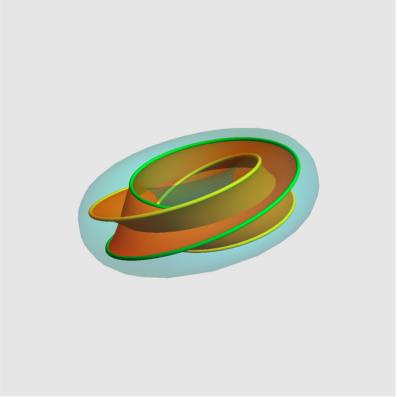}
\includegraphics[height=6cm,width=6cm]{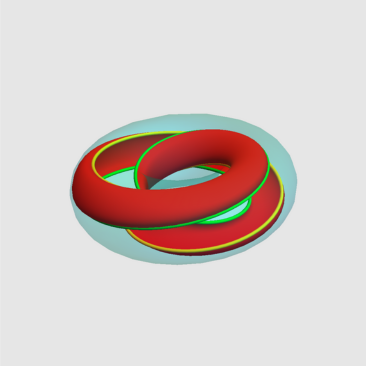}
\caption{\small On the left: the hyperbolic region bounded by the two quasi-umbilical curves. 
On the right: the elliptic  region  bounded by the two quasi-umbilical curves.}\label{FIG3}
\end{center}
\end{figure}

\subsection{Isothermic maps}\label{s2.1.bis} 
Viewing the Einstein universe $\EE$ as a
symmetric $R$-space, i.e., as the orbit $\mathcal{O}$ of the height 1 parabolic subalgebra
$\mathfrak{h}_0$ of $\aa$ (cf. Section \ref{models-of-EU}), we now discuss isothermic 
maps in $\EE$.
According to \cite{BDPP},
the notion of isothermic map is strictly related to the structure of symmetric $R$-space of the target space.
%
Let $f : X \to \EE\cong \mathcal{O}$ be a timelike immersion. Consider the vector bundle 
$$
   \mathfrak{H}^{\perp}_f = \left\{ (x, \mathbf{h})\in X\times \aa \mid \mathbf{h} \in \mathfrak{h}^{\perp}_{x}\right\},
   $$
where, for each $x\in X$, ${\mathfrak{h}}_{x}$ denotes the 
Lie algebra of the stabilizer of {$f({x})$}.

\begin{defn}
A timelike immersion $f : X \to \EE$ is \textit{isothermic} if there is a nonzero 
1-form 
$\delta \in \Omega^1(X)\otimes \mathfrak{H}^{\perp}_f$
which is closed as an element of $\Omega^1(X)\otimes \aa$.
Thus $\delta$ takes values in
$\mathfrak{H}^{\perp}_f$ and is closed when viewed as an $\aa$-valued form.
%
The 1-form $\delta$ is called an \textit{infinitesimal 
deformation} of $f$.
\end{defn}

\begin{remark}\label{r.s2.1}
The concept of isothermic map originates from the deformation problem of submanifolds in homogeneous spaces 
(cf. \cite{BHJPR,Ca2, JeLN, JeDg, Mu, MNTo,MNSIGMA,Pe}). 
Two maps $f:{M} \to G/H$ and $\hat{f}:\hat{{M}} \to G/H$ 
are said to be $k$\textit{th order deformations of each other} (with unfixed parameters) if there exist a diffeomorphism 
$\Phi:{M}\to \hat{{M}}$, the {\it change of parameters}, and a nonconstant map $\Delta: {M}\to G$, the \textit{deformation}, 
such that $f$ and $\Delta(p) \cdot\hat{f}\circ \Phi$ have $k$th order contact at $p$, for every $p\in M$. 
In addition, if $M=\hat{{M}} $ and $\Phi=\mathrm{Id}_{{M}}$, then $f$ and  $\hat{f}$ are said to be {\it $k$th order deformations of
each other} (with fixed parameters). If $G/H$ is a symmetric $R$-space and if $f:{M}\to G/H$ is isothermic, 
then $f$ admits second order deformations \cite{BDPP}. 
In the case under discussion, the situation is the following. Let $f:X\to \EE$ be an isothermic timelike 
immersion, $p:\tilde{X}\to X$ be a simply connected
covering of $X$, and $\delta$ be an infinitesimal deformation of $f$. Since {$\mathfrak{h}_x^{\perp}$} 
is abelian, $\delta$ satisfies the Maurer--Cartan equation. So, there exists a smooth map 
$\Delta : \tilde{X}\to \A$, such that $\Delta^{-1}d\Delta= p^*\delta$. 
Then, 
$$
  \hat{f} :  \tilde{X} \ni \tilde{x}\longmapsto \Delta(\tilde{x})\cdot f|_{p(\tilde{x})}\in \EE
    $$ 
is a second order deformation (with fixed parameters) of $f\circ p$. 
The deformability of $f$ depends on the existence of an infinitesimal deformation originating a map 
$\Delta: \tilde{X}\to \A$ which is invariant under the group of deck transformations of the covering. 
\end{remark}

\subsection{General properties of quasi-umbilical immersions}\label{s2.2}

\begin{lemma}
Let  $f:X\to \EE$ be a quasi-umbilical immersion and let $g$ denote the induced Lorentzian metric 
$f^*(g_{{\mathcal E}})$. 
Then, the conformal Gauss map $\NN$ has rank 1. In addition, 
$\mathrm{Im}\,d{\mathcal N}_{f}$ is a null line bundle such that 
$\mathrm{Im}\,d{\mathcal N}_{f}$ $=$ $\mathrm{Ker}\,d{\mathcal N}_{f}$.\footnote{Notice that $d{\mathcal N}_{f}|_{{ x}}$ is a 
2-step nilpotent self-adjoint endomorphism 
of ${T}_x(X)$, $\forall \, {x}\in X$.}
\end{lemma}

\begin{proof}
Since the property is local in nature, we may assume the existence of a global cross section $F:X\to {\mathcal F}_2(f)$,
such that $[{F}_0]=f$. Then, from \eqref{2.2},  \eqref{2.2.bis} and  \eqref{2.2.tris}, we compute
$$
   d {F}_0= {\phi^0_0 F_0} +\phi^1_0 { F}_1+\phi_0^2{ F}_2,\quad 
    d\NN=h(\phi^1_0+\eta \phi^2_0)({ F}_1-\eta { F}_2)+\phi_4^3{ F}_0,
     $$
where $h$ is a nowhere vanishing smooth function and $\eta = \pm 1$.  
Let $(\partial_1,\partial _2)$ denote the trivialization of ${T}(X)$ dual to $(\phi^1_0,\phi^2_0)$. Then, 
identifying  $T(X)$ with the quotient bundle $\span\{ {F}_0,{ F}_1,{ F}_2\}/\span\{{F}_0\}$, we obtain
$$
    d{\mathcal N}_{f}=h(\phi^1_0+\eta \phi^2_0)(\partial_1-\eta \partial_2).
     $$
This implies that $\mathrm{Im}\,d{\mathcal N}_{f}= \mathrm{Ker}\,d{\mathcal N}_{f}=\span\{\partial_1-\eta \partial_2\}$. 
From the expression of $g=-(\phi^1_0)^2+(\phi^2_0)^2$, it follows that $\mathrm{Im}\,d{\mathcal N}_{f}$ is a null line subbundle 
of ${T}(X)$. This concludes the proof.
\end{proof}

\begin{defn}
On $X$ there is a unique \textit{canonical orientation} such that $(V,V_*)$ is a positive basis of ${T}_x(X)$, for every ${x}\in X$, 
for every nonzero element $V$ of $\mathrm{Im}({d{\mathcal N}_{f}}|_{{x}})$, and for every null vector $V_*\in {T}_x(X)$, such that 
$g(V,V_*)=-1$. From now on, we consider on $X$ the canonical orientation induced by $f$. By construction, the map 
$$
    \Psi :  \ima d{\mathcal N}_{f}\setminus\{0\}\ni \mathbf{V}  \longmapsto    -g(d{\mathcal N}_{f}(\mathbf{V}_*),\mathbf{V}_*)\in \R
       $$
is well defined and is either strictly positive or strictly negative. The sign of $\Psi$, denoted by $\varepsilon_f$,  is called the \textit{helicity} 
of $f$.
\end{defn}

\begin{remark}
If $X$ is equipped with the canonical orientation induced by $f$ and if $F$ is a 
second order frame field along $f$, then $h_{11}=h_{12}$ and $h:=h_{11}$ is nowhere vanishing. 
The sign of $h$ is equal to the helicity of $f$. Replacing $f$ by $-f$, the helicity changes sign.
\end{remark}

\begin{defn}
Let $f:X\to \EE$ be a quasi-umbilical immersion. A second order frame field along $f$ is said 
to be \textit{adapted} if
\begin{equation}\label{2.1ada} 
   h = \varepsilon_f,\quad {\phi^3_4} 
   =\phi^0_0-2\phi^2_1=0,\quad (\phi^1_4+\phi^2_4)\wedge (\phi^1_0+\phi^2_0)=0.
   \end{equation}
\end{defn}

The following result can be found differently formulated in \cite[p. 103]{THE}.

\begin{prop}
Adapted frame fields exist near any point of $X$. In addition, if $F, \hat{F} : U\to \A$ are {adapted}
frame fields, then
$\hat{F} = F\,R$, where $R: U \to H_*$ is a smooth map into the abelian 1-dimensional closed subgroup 
 $$
     {H}_*=\{R(s) = R(r,s,x,0) \in {H}_2 \mid  r=2s, \, x=0\}.
         $$
The adapted frame fields along $f$ are the local sections of a reduced subbundle
${\mathcal F}_*(f)$ of ${\mathcal F}_2(f)$, with 
structure group ${ H}_*$.
\end{prop}

\begin{defn}
The map $ {\mathcal F}_*(f) \ni (x,F)\mapsto [F_4]\in \EE$ is constant along the fibers. It descends to a 
map $f^{\sharp}:X\to \EE$, 
called the \textit{dual} 
of $f$. 
\end{defn}

\begin{remark}
The dual map is the second envelope of the 1-parameter family of central tori of $f$. That is, 
$$
 f^{\sharp}(x)\in \partial {\mathcal A}|_{x},\quad 
df^{\sharp}|_{x}({T}_{x}X)\subset {T}_{f^{\sharp}(x)}(\partial {\mathcal A}|_{x}),\quad \forall\, x\in X,
     $$ 
where $\partial {\mathcal A}|_{x}$ is the AdS wall of $\NN (x)$.
\end{remark}

\begin{defn}
A quasi-umbilical immersion $f$ is called \textit{regular} if the {leaf} space $\Lambda$ 
of the vertical distribution $\mathrm{Ker}\,d\mathcal{N}_f$
 is Hausdorff with respect to the quotient topology. If $f$ is regular, $\Lambda$ is a connected 
 1-dimensional manifold and the natural projection $\pi : X\to \Lambda$ is a smooth submersion with 
 connected fibers. 
 Then, $\NN=\Gamma\circ \pi$, where $\Gamma:\Lambda \to \SS$ is a null immersed curve, called the \textit{directrix curve} 
 of $f$. Note that locally every quasi-umbilical immersion is regular.
\end{defn}

\begin{remark}
As we will show below, a regular quasi-umbilical immersion can be reconstructed from its directrix curve. 
This explain the key role played by null curves in the neutral space form $\SS$. 
By the identification of $\SS$ with $\mathbb{V}_-^2\oplus {S}^2_+$ described in 
\eqref{S22}, a null curve of $\SS$  has a parametrization of the form 
$\Gamma = \alpha+\sqrt{1-\langle \alpha,\alpha\rangle}\beta$, 
where $\alpha:\R \to \mathbb{V}_-$ is an immersed plane curve parametrized by arclength  and $\beta:\R \to  {S}^2_+$ is a 
spherical curve with speed
$$
\upsilon_{\beta}= \frac{\sqrt{1-\langle \alpha,\alpha \rangle \sin^2(\theta)}}{1-\langle \alpha,\alpha\rangle},
    $$
where $\theta=\widehat{\alpha \alpha'}$. Thus, quasi-umbilical surfaces depend on two functions in one variable,  
the geodesic curvatures of the curves $\alpha$ and $\beta$.
\end{remark}
 
\begin{thmx}\label{thmA}
The conformal Gauss map $\NN:X\to \SS$ of a quasi-umbilical immersion $f:X\to \EE$ is harmonic.
Moreover, if $f$ is regular, then
$f$ is isothermic and its infinitesimal deformations depend on one arbitrary function in one variable.
\end{thmx}

\begin{proof}
The first part of the theorem  is local in nature, so we assume the existence of a global adapted frame field 
$F:X\to \A$, such that $[{F}_0]=f$. We then have
\begin{equation}\label{s2.2.1}
\begin{cases}
g=-(\phi^1_0)^2+(\phi^2_0)^2 = f^*(g_{\mathcal{E}}), \\
\phi^3_0={\phi^3_4}=\phi^0_0-2\phi^2_1=0,\\
\phi_1^3=  {\phi_2^3} 
=\varepsilon_f(\phi^1_0+\phi^2_0),\\
\phi_4^1=r(\phi^1_0+\phi^2_0)+\alpha,\\
\phi_4^2=r(\phi^1_0+\phi^2_0)-\alpha
\end{cases}
\end{equation}
 where $r$ is a smooth function and $\alpha$ is an exterior differential 1-form. Then
\begin{equation}\label{s2.2.2}
    d\NN=\phi^3_1{ F}_1-\phi^3_2 { F}_2 = \varepsilon_f(\phi^1_0+\phi^2_0)({ F}_1-{ F}_2).
      \end{equation}
Since $X$ is 2-dimensional, the harmonicity of $\NN$ only depends on the conformal structure induced by $f$ on $X$. 
Thus, it suffices to compute the tension field of $\NN$ with respect to the Lorentzian metric $g$.  
Let $\nabla$ be
the Levi-Civita covariant derivative of $g$ and $\tilde{\nabla}$ be the covariant derivative on the pullback bundle $ \NN^*({\it T}(\SS))$ 
induced by the Levi-Civita connetion of $\SS$. Next, let $D$ be the covariant derivative on ${\it T}^*(X)\otimes \NN^*({\it T}(\SS))$ 
induced by $\nabla$ and $\tilde{\nabla}$. Let $\eta^1=\phi^1_0+\phi^2_0$ and $\eta^2=\phi^1_0-\phi^2_0$. 
Then $\phi^2_1=a\eta^1+b\eta^2$, where $a,b$ are smooth functions. We then have
$$
   \begin{cases}\tilde{\nabla}({ F}_1-{ F}_2)=\phi^2_1({ F}_1-{ F}_2)=(a\eta^1+b\eta^2)({ F}_1-{ F}_2),\\
     \nabla \eta^1=(-3 {a}\eta^1+b\eta^2)\otimes \eta^1.
      \end{cases}
        $$
This implies
$$
    D(d\NN)=-2a(\eta^1\otimes \eta^1)\otimes  ({ F}_1-{ F}_2).
       $$
Taking into account that $2g=\eta^1\otimes \eta^2$, we may conclude that the tension field ${\tr}_g(D(d\NN))$ 
is identically zero. This proves the first claim.

Next, suppose that $f$ is regular. Let $\Gamma : \Lambda \to \SS$ be the directrix curve of $f$ and $\zeta$ be 
a nowhere vanishing 1-form on $\Lambda$. We put $\tilde{\zeta}=\pi^*(\zeta)$. Then there exists a unique adapted 
frame field $F:X\to \A$ such that $\tilde{\zeta}=\pm (\phi^1_0+\phi^2_0)$. Possibly replacing $\zeta$ with $-\zeta$, 
we may assume that  $\tilde{\zeta}=\phi^1_0+\phi^2_0$. We call $\zeta$ a \textit{line element} of $\Lambda$.
Then, 
\begin{equation}\label{s2.2.3}
{\begin{split} \phi &=  2\phi^2_1M^0_0+  \phi^2_1M^2_1+\phi^1_0M^1_0+\phi^2_0M^2_0
  +\varepsilon \tilde{\zeta}(M^3_1+M^3_2)\\
 &\qquad  +r\tilde{\zeta}(M^1_4+M^2_4)+\alpha(M^1_4-M^2_4),\end{split}}
     \end{equation}
%
%
 where $(\phi^1_0,\phi^2_0)$ is a positive-oriented pseudo-orthogonal coframe, $\varepsilon = \pm 1$ 
 is the helicity of $f$, 
 $r$ is a smooth function, and $\alpha$ is a 1-form. From the structure equations, we obtain
\begin{equation}\label{s2.2.4}
  0=d\tilde{\zeta} = d(\phi^1_0+\phi^2_0)=\phi^2_1\wedge (\phi^1_0+\phi^2_0)=\phi^2_1\wedge \tilde{\zeta},
   \end{equation}
which implies
\begin{equation}\label{s2.2.5}
     \phi^2_1=s\tilde{\zeta},
        \end{equation}
where $s:X\to \R$ is a smooth function. 
Let $\ell : \Lambda\to \R$ be a nonzero smooth function. Put $\tilde{\ell} =\ell\circ \pi$ and denote by 
$\eta$ the 1-form
\begin{equation}\label{s2.2.6}
    {\eta =  \tilde{\ell} \tilde{\zeta}(M^1_4-M^2_4).}
       \end{equation}
This is a closed 1-form with values in the polar space ${\mathfrak h}_0^{\perp}$ of the parabolic subalgebra ${\mathfrak h}_0$. 
It follows from \eqref{s2.2.3} and \eqref{s2.2.4} that
 $$
     \phi\wedge \eta + \eta\wedge \phi = 0.
        $$ 
Next, let $\delta := F\cdot \eta\cdot F^{-1}$. Then,
$$
    d\delta = dF\wedge\eta F^{-1}-F\cdot\eta\wedge dF^{-1} = F\cdot(\phi\wedge \eta + \eta\wedge \phi)\cdot F^{-1}=0.
       $$
{Since ${\mathfrak{h}}_{x}=Ad_{F(x)}({\mathfrak h}_0)$ and  ${\mathfrak{h}}^{\perp}_{x}=Ad_{F(x)}({\mathfrak h}^{\perp}_0)$,  
then $\delta|_{x}\in  {T}_x^*(X)\otimes \mathfrak{h}_x^{\perp}$, for every $x\in X$. 
This proves that $\delta$ is an infinitesimal deformation.}

Every infinitesimal deformation arises in this way. In fact, if $\delta$ is an infinitesimal deformation, then $\eta =Ad_{F^{-1}}\delta$ 
is a ${\mathfrak{h}}_0^{\perp}$-valued $1$-form. From $d\delta = 0$ it follows that  $d\eta =\eta\wedge \phi + \phi\wedge \eta$. 
Using \eqref{s2.2.3} it is easy to see that the latter equation holds true if and only if $\eta$ is as in \eqref{s2.2.6}, for some 
nonzero smooth function $\ell:\Lambda\to \R$. This concludes the proof.
\end{proof}

\begin{figure}[ht]
\begin{center}
\includegraphics[height=6cm,width=6cm]{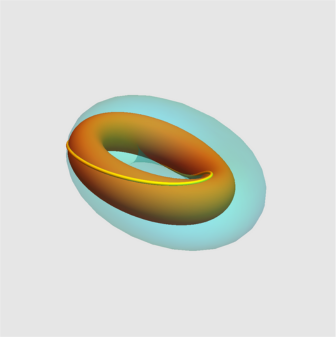}
\includegraphics[height=6cm,width=6cm]{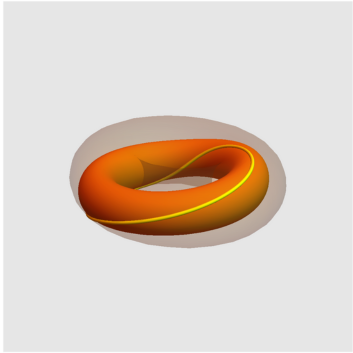}
\caption{\small The timelike immersed tube corresponding to the ``biisotropic circle'' \eqref{bii-circle}: on the left, the portion of the surface contained 
in the positive AdS chamber (a parabolic component); on the right, the portion contained in the negative AdS chamber 
(the other parabolic component). The yellow curve is the curve of umbilical points of the surface. 
}\label{FIG4}
\end{center}
\end{figure}

\begin{defn}
Let $f:X\to \EE$ be a quasi-umbilical immersion and $f^{\sharp}:X\to \EE$ be its dual map. 
According to whether the quadratic form $\langle d\NN,d\df \rangle$ is identically zero or not,
$f$ is called \textit{exceptional} or \textit{general}.
 \end{defn}

\section{Quasi-umbilical immersions: the exceptional case}\label{s3}

\subsection{Biisotropic curves in the neutral space form $\SS$}\label{s3.1} 

Let $\Lambda$ be a connected 1-dimensional manifold and $\zeta$ be a nowhere zero 1-form on $\Lambda$. 
For a smooth map $G:\Lambda \to \R^k$, we denote by $G', G''$, etc., the derivatives of $G$ with respect to $\zeta$.

\begin{defn}
An immersed curve $\Gamma:\Lambda\to \SS$ is called \textit{biisotropic} if $\langle \Gamma',\Gamma'\rangle = \langle \Gamma'',\Gamma''\rangle = 0$ and $(\Gamma'\wedge \Gamma'')|_{\tau}\neq 0$, for every $\tau\in \Lambda$.
\end{defn}

\begin{prop}\label{prop:except-biisotropic}
Let $f:X\to \EE$ be a regular quasi-umbilical immersion. 
Then, $f$ is exceptional if and only if its directrix curve $\Gamma$ is biisotropic.
\end{prop}

\begin{proof}
Consider a line element $\zeta$ on $\Lambda$ and the unique adapted frame field $F$ along $f$ such that 
$\phi^1_0+\phi^2_0=\tilde{\zeta} = \pi^*(\zeta)$ (cf. the proof of Theorem \ref{thmA}). 
Now, the dual map $\df=\varrho { F} _4$, where $\varrho$ is a smooth positive function. From \eqref{s2.2.3} and  \eqref{s2.2.5}, 
we compute
\begin{equation}\label{s3.1.1} 
 d\NN=\varepsilon \tilde{\zeta}({ F}_1-{ F}_2),\quad 
    d{ F}_4=r\tilde{\zeta}({F}_1 {+} {F}_2) + \alpha({ F}_1-F_2)-2s \tilde{\zeta}{ F}_4,
     \end{equation}
which implies $\langle d\df, d\NN\rangle = -2\varepsilon \varrho r {(\tilde{\zeta})^2}$. 
Therefore, $f$ is exceptional if and only if $r=0$. 
On the other hand, $\Gamma'\circ \pi = \varepsilon ({ F}_1-{ F}_2)$. Differentiating the latter identity yields
$$
  (\Gamma''\circ \pi)\tilde{\zeta} = d(\Gamma'\circ \pi) = \varepsilon d ({F}_1-{F}_2)= 
       - \varepsilon \tilde{\zeta}\left(2r{ F}_0+s({ F}_1-{ F}_2) + {F}_4\right).
    $$
Thus, $\langle \Gamma''\circ \pi,\Gamma''\circ \pi\rangle = {- 4r \varepsilon^2}
$, 
from which the claim follows.
\end{proof}

\begin{lemma}
Let $\Gamma:\Lambda\to \SS$ be a biisotropic curve. Then there exist an isotropic (null)
2-dimensional subspace ${\mathbb V}_{\Gamma}\subset \MM$ and 
a 3-dimensional subspace ${\mathbb W}_{\Gamma}\subset \MM$, such that 
\begin{equation}\label{s3.1.2}
 [(\Gamma'\wedge \Gamma'')|_{\tau}]={\mathbb V}_{\Gamma}\subset
   [(\Gamma\wedge\Gamma'\wedge \Gamma'')|_{\tau}]=
      {\mathbb W}_{\Gamma},\quad \forall \,\tau\in \Lambda.
       \end{equation}
\end{lemma}

\begin{proof}
Differentiating the scalar product relations
\begin{equation}\label{s3.1.3} 
  \langle \Gamma,\Gamma\rangle = 1,\quad \langle \Gamma',\Gamma'\rangle = \langle \Gamma'',\Gamma''\rangle = 0
    \end{equation}
yields the following relations
\begin{equation}\label{s3.1.4} 
\langle \Gamma,\Gamma'\rangle=\langle \Gamma,\Gamma''\rangle = \langle \Gamma,\Gamma'''\rangle=\langle \Gamma',\Gamma''\rangle
={\langle \Gamma',\Gamma'''\rangle}
   = \langle \Gamma'',\Gamma'''\rangle = 0.
       \end{equation}
It follows from \eqref{s3.1.3} that $[(\Gamma'\wedge \Gamma'')|_{\tau}]$ is a null 2-dimensional subspace and that 
$[(\Gamma\wedge\Gamma'\wedge \Gamma'')|_{\tau}]$ is a 3-dimensional subspace of type $(+,0,0)$, for every $\tau\in \Lambda$. 
By construction, $[(\Gamma'\wedge \Gamma'')|_{\tau}]= [(\Gamma\wedge\Gamma'\wedge \Gamma'')|_{\tau}]^{\perp}$.  
From \eqref{s3.1.4} it then follows that $\Gamma'''|_{\tau}\in [(\Gamma'\wedge \Gamma'')|_{\tau}]$, for every $\tau$. 
This implies that the 2-dimensional subspace 
$[(\Gamma'\wedge \Gamma'')|_{\tau}]$ and 
the 3-dimensional subspace $[(\Gamma\wedge\Gamma'\wedge \Gamma'')|_{\tau}]$ are constant, as claimed.
\end{proof}

The set of unit vectors lying in ${\mathbb W}_{\Gamma}$ is the union of two disjoint affine planes 
parallel to ${\mathbb V}_{\Gamma}$.  
The trajectory of $\Gamma$ is contained in one of them, denoted by ${\mathcal V}_{\Gamma}$. 
The positive orientation of ${\mathbb V}_{\Gamma}$ induces on ${\mathcal V}_{\Gamma}$  the 
structure of an affine oriented plane and 
$\Gamma$ can be considered as an affine plane curve.  Let ${dA}$ be a positive element of $\Lambda^2({\mathbb V}_{\Gamma})$. 
Since $\Gamma$ has no flex points (i.e., points where ${\Gamma'}\wedge{\Gamma''} = 0$), 
there exists a unique 1-form $\zeta$, such that ${dA}(\Gamma',\Gamma'')=1$,  the \textit{affine line element} of 
$\Gamma$ relative to ${dA}$. If we fix an affine line element, then $\Gamma\wedge \Gamma'\wedge \Gamma''$ 
and $\Gamma'\wedge \Gamma''$ are constant. Differentiating ${dA}(\Gamma',\Gamma'')=1$, we get ${dA}(\Gamma',\Gamma''')=0$. 
Then, $\Gamma'''=h\Gamma'$, where $h:\Lambda\to \R$ is a smooth function, the \textit{affine curvature} of 
$\Gamma$ relative to ${dA}$. 

\begin{remark}[Curvature sign]\label{r:curv-sign}
Our definition of the affine curvature is not the only reasonable possibility; one can also change its sign,
defining curvature to be $\tilde{h} = -h$, as in the classical treatise of Blaschke \cite{BlaschkeII}.
\end{remark}

\begin{defn}
A lightcone basis $\mathfrak{C}=({C}_0,\dots,{C}_4)$ is said to be \textit{calibrated} to $\Gamma$ 
and $\zeta$ if
$$
    \Gamma'\wedge \Gamma''={C}_4\wedge ({C}_1-{C}_2), \quad \Gamma\wedge \Gamma'\wedge \Gamma''
       ={C}_3\wedge {C}_4 \wedge ({C}_1-{C}_2).
           $$
\end{defn}

\noindent If  $\mathfrak{C}$ is a calibrated basis, then
$$
    \Gamma = {C}_3+x({C}_1- {{C}_2})+y{C}_4,
       $$
 where $\gamma=(x,y):\Lambda \to \R^2$ is an affine plane curve ($x'y''-x''y'=1$), 
 the \textit{affine reduction} of $\Gamma$ relative to $\mathfrak{C}$.  Let
 $$
     \upsilon = \sqrt{(x')^2+(y')^2} \quad \text{and} \quad \mu = xy'-x'y 
        $$
be the speed and the angular momentum of $\gamma$. The vector field
 \begin{equation}\label{s3.1.5.0}
    {\rm R}=-x'{C}_0+\frac{1}{2}y'({C}_1+{C}_2)+\frac{\mu^2}{2\upsilon^2}y'({C}_1-{C}_2)+\mu{C}_3-
        \frac{\mu^2}{2\upsilon^2}x'{\rm C}_4\end{equation}
 is called the \textit{isotropic normal vector field} along $\Gamma$ relative to $\mathfrak{C}$. By construction, we have
 \begin{equation}\label{s3.1.5} 
    \langle {\rm R},{\rm R}\rangle = \langle {\rm R},\Gamma\rangle =\langle {\rm R},\Gamma'\rangle =0,\quad 
      \langle {\rm R},\Gamma''\rangle = 1.
            \end{equation}
 If ${\rm R}$ and $\hat{{\rm R}}$ are isotropic normal vector fields relative to different area elements and 
 different calibrated bases, then $\tilde{{\rm R}}=c{\rm R}+r {\Gamma'}$,
 where $c$ is a positive constant and $r$ is a smooth function. Then, the null plane $[({\rm R}\wedge \Gamma')|_{\tau}$ is 
 independent of the choice of the null normal vector field. 
 
 \begin{defn}
 The \textit{normal tube} of $\Gamma$ is the circle bundle over $\Lambda$ given by
\begin{equation}\label{s3.1.6}
 \pi_{{\Lambda}}: {\mathbb T_\Gamma}=\left\{(\tau,[{V}])\in \Lambda\times \EE \mid {V}\in 
    [({\rm R}\wedge \Gamma')|_{\tau} \right\}\to    \Lambda.
    \end{equation}
If ${\rm R}$ is an isotropic normal vector, the elements of the fiber $\pi_{\Lambda}^{-1}({\tau})$ can 
be described by
$[{V}(\tau,\theta)] = [\cos(\theta){\rm R}|_{\tau}+\frac{1}{2}\sin(\theta)\Gamma'|_{\tau}]$, for $\theta\in \R/2\pi\Z$. 
Hence, $\mathbb{T}_\Gamma$ can be identified with $\Lambda\times S^1$ by means of the mapping
\begin{equation}\label{s3.1.7} 
  \Lambda\times S^1\ni (\tau,\theta) \mapsto (\tau,[{V}(\tau,\theta)])\in \mathbb{T}_\Gamma.
     \end{equation}
\end{defn}
Let 
\begin{equation}\label{s3.1.8}
  {\rm C}_{\pm} =\{(\tau,[{V}(\tau,\pm \pi/2)])   \mid \tau\in \Lambda\} \subset \mathbb{T}_\Gamma,
    \end{equation}
    and call them the \textit{umbilical curves} of $\mathbb{T}_\Gamma$.
The complement of ${\rm C}_{+}\cup {\rm C}_{-}$ consists of the two connected open sets
\begin{equation}\label{s3.1.9}
  \mathbb{T}^{\pm}_\Gamma =\{(\tau,[{V}(\tau,\theta)]) \mid \tau\in \Lambda, \, \sgn(\cos\theta)=\pm 1\},
    \end{equation}
referred to as the \textit{positive and negative parabolic components} of $\mathbb{T}_\Gamma$. 

\begin{figure}[ht]
\begin{center}
\includegraphics[height=6cm,width=6cm]{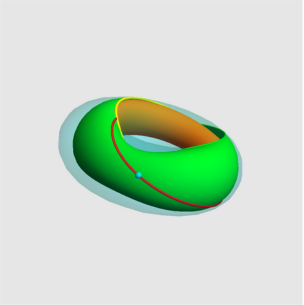}
\includegraphics[height=6cm,width=6cm]{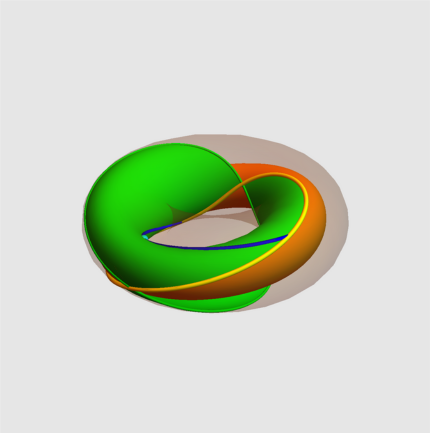}
%
\caption{\small The positive (left) and negative (right) {parabolic cylinders} 
$f_\Gamma^+({\mathbb T}^{+}_\Gamma)$ and $f_\Gamma^-({\mathbb T}^{-}_\Gamma)$
and (in green) the central torus of $f_\Gamma({\mathbb T}_\Gamma)$ for
the tube ${\mathbb T}_\Gamma$
along the ``biisotropic circle'' \eqref{bii-circle}. On the left, the positive part of the 
central torus; on the right, its negative part. 
The red and blue curves are the positive and negative parts 
of the line of tangency 
of the osculating torus with $f_\Gamma({\mathbb T}_\Gamma)$. 
The {yellow line} is the umbilical line. 
}\label{FIG5}
\end{center}
\end{figure}

\begin{defn}\label{def:taut-map}
We call 
$f_{\Gamma}:\mathbb{T}_\Gamma \to \EE$, $(\tau,[{V}]) \longmapsto [{V}]$,
the \textit{tautological map of} $\Gamma$.
The restrictions of $f_{\Gamma}$ to $ \mathbb{T}^{\pm}_\Gamma$ are denoted by $f^{\pm}_{\Gamma}$, respectively.
%
\end{defn}

\subsection{The structure of exceptional quasi-umbilical surfaces}\label{s3.2} 

Using the notation and the constructions of the previous section,
we are now in a position to describe the geometric structure of exceptional quasi-umbilical surfaces. 
We can prove the following.

\begin{thmx}\label{thmB}
Let $\Gamma:\Lambda\to \SS$ be a biisotropic curve. Then, the tautological map $ f_{\Gamma}$ is a timelike immersion 
with conformal Gauss map $\Gamma\circ \pi_{{\Lambda}}$  and umbilic locus ${\rm C}_{+}\cup {\rm C}_{-}$. In addition,
$f^{\pm}_{\Gamma}$  are exceptional quasi-umbilical immersions with helicities $\pm1$ and
dual maps $(f^{\pm}_{\Gamma})^{\sharp}=[\mp \Gamma''\circ  \pi_{{\Lambda}}  ]$.
Conversely, let $f:X\to \EE$ be a regular quasi-umbilical immersion with helicity $\varepsilon=\pm 1$. 
If $f$ is exceptional, then 
{$f(X)\subset f^{\pm}_{\Gamma}\left({\mathbb T}^{\pm}_{\Gamma}\right)$},
where $\Gamma$ is the 
directrix curve of $f$.
\end{thmx}

\begin{proof}
Let $\mathbb{T}_\Gamma$ be identified with
$\Lambda\times S^1$ as in \eqref{s3.1.7}. Take a positive area element ${dA}$ on ${\mathbb V}_{\Gamma}$ and denote by $\zeta$
the corresponding affine arc element.  Consider a calibrated lightcone basis $\mathfrak{C}$ and the associated null normal 
vector field ${\rm R}$ (cf. \eqref{s3.1.5.0}). 
Let $h$ be the affine curvature of $\Gamma$. Let ${\rm T}: \Lambda\times S^1 \to \MM$ be the lift of  $f_{\Gamma}$ 
defined by ${\rm T}(\tau,\theta)= \cos(\theta){\rm R}|_{\tau}+\frac{1}{2}\sin(\theta)\Gamma'|_{\tau}$.
Then
$$
  d{\rm T}= \Big(-\sin(\theta){\rm R}+\frac{1}{2}\cos(\theta)\Gamma'\Big)d\theta
    +\Big(\cos(\theta){\rm R}'+\frac{1}{2}\sin(\theta)\Gamma''\Big)\zeta.
   $$
Differentiating \eqref{s3.1.5}, 
{taking into account that $\Gamma'''=h\Gamma'$},
yields $\langle {\rm R},{\rm R}'\rangle = \langle {\rm R}',\Gamma''\rangle = 0$, and
$\langle {\rm R}',\Gamma'\rangle = -1$, 
from which one computes
$$
  {\langle {d{\mathrm T}},{d{\mathrm T}}\rangle=-\zeta\left({d}\theta-(\cos\theta)^2\langle {\rm R}',{\rm R}'\rangle\zeta\right).}
  $$
This implies that the tautological map $f_{\Gamma}$ is a timelike immersion.

Let ${F}^{\pm}_0: {\mathbb T}_{\Gamma}^{\pm}\to \A$ be the lift of $ f^{\pm}_{\Gamma}$ defined by
\begin{equation}\label{s3.2.0.1}
     {F}^{\pm}_0 = \frac{\pm 1}{\cos \theta} {\rm T}= \pm\Big({\rm R}+\frac{1}{2}\tan(\theta)\Gamma'\Big).
         \end{equation}
Denote by $(\partial_{1},\partial_{2})$  the trivialization of ${T}({\mathbb T}_\Gamma^{\pm})$ dual to 
the coframe $(\zeta,{du})$, where $u=\tan\theta$. 
Consider the maps ${F}^{\pm}_1,\dots,{F}^{\pm}_4: \mathbb{T}^{\pm}_\Gamma\to \MM$, defined by
\begin{equation}\label{s3.2.0.2}\begin{cases}
 {F}^{\pm}_1=-\partial_1 {F}^{\pm}_0-(\langle \partial_1  {F}^{\pm}_0,\partial_1 {F}^{\pm}_0\rangle+1)\partial_2 {F}^{\pm}_0,\\
{F}^{\pm}_2=-\partial_1 {F}^{\pm}_0-(\langle \partial_1 {F}^{\pm}_0,\partial_1 {F}^{\pm}_0\rangle-1)\partial_2 {F}^{\pm}_0,\\
{F}^{\pm}_3=\Gamma\circ \pi_{{\mathbb T}},\\
{F}^{\pm}_4=\mp\Gamma''\circ \pi_{{\mathbb T}}.
      \end{cases}\end{equation}
Then,  ${F}_{\pm}=({F}^{\pm}_0,\dots,{F}^{\pm}_4)$ is an adapted frame field along $ f^{\pm}_{\Gamma}$. 
The 1-form 
$\phi_{\pm}={F}_{\pm}^{-1}{d}{F}_{\pm}$ is as in 
\eqref{s2.2.3}, with $\phi^2_1=0$, $r=0$, $\alpha=-h\zeta$, and
\begin{equation}\label{s3.2.0.3}
 \begin{cases}
  \phi^1_0=-\frac{1}{2}\left(du +(\frac{\mu^2+(xx'-yy')^2}{\upsilon^4}+1)\zeta  \right),\\
   \phi^2_0=\frac{1}{2}\left(du +(\frac{\mu^2+(xx'-yy')^2}{\upsilon^4}-1)\zeta  \right),
      \end{cases}
        \end{equation}
where $\gamma=(x,y)$ is the affine reduction of $\Gamma$ with respect to $\mathfrak{C}$. 
This proves that  $f^{\pm}_{\Gamma}$ 
are exceptional quasi-umbilical immersions with helicity $\varepsilon=\pm1$, conformal Gauss maps $\Gamma\circ \pi_{{\Lambda}}$,
and dual maps  $[-\varepsilon \Gamma''\circ \pi_{{\mathbb T}}]$. By continuity, $\Gamma\circ \pi_{{\Lambda}}$ is the conformal Gauss 
map of $f_{\Gamma}$. 
From the identity {$(d\NN\wedge {\rm T})|_{(\tau,\theta)}=\cos(\theta)({\rm R}\wedge \Gamma')|_{\tau}\zeta$}, it follows
that $(\tau,\theta)$ is an umbilical point if and only if $(\tau,\theta)\in {\rm C}_{+}\cup {\rm C}_-$.

Conversely, let $f:X\to \EE$ be a regular, quasi-umbilical immersion, with directrix curve $\Gamma : \Lambda\to \SS$, 
and let ${\mathcal F}_*(f)$ be the bundle of adapted frames along $f$. Let ${\mathfrak N}_2^+(\MM)$ denote the 
Grassmannian of {positive} null planes of $\MM$. 
The map
$$ 
     {\mathcal F}_*(f)\ni (\tau,F) \longmapsto [{F}_4\wedge ({F}_1-{F}_2)]\in {\mathfrak N}_2^+(\MM)
       $$ 
is constant on the fibers of ${\mathcal F}_*(f)$. So, it descends to a smooth map ${\mathbb V}_f:X\to {\mathfrak N}_2^+(\MM)$. 
Let $F:U\subset X\to \A$ be a local section of ${\mathcal F}_*(f)$. 
The 1-form $\phi=F^{-1}dF$ is as in \eqref{s2.2.3}, with $r=0$ and $\alpha=-q\zeta$, for some smooth function $q$.  
Then ${d} ({F}_4\wedge ({F}_1-{F}_2))=-3\phi^2_1  ({F}_4\wedge ({F}_1-{F}_2))$. 
This implies that $ [{F}_4\wedge ({F}_1-{F}_2)]$ is constant. 
By construction, ${\mathbb V}_f|_{U}=[{F}_4\wedge ({F}_1-{F}_2)]$. 
This proves that ${\mathbb V}_f$ is a constant positive null plane. 
Next, we choose a positive area element ${dA}\in \Lambda^2({\mathbb V}_f)$ and consider 
the unique global cross section $F: X\to \A$ of ${\mathcal F}_*(f)$, such that ${dA}({F}_4,{F}_1-{F}_2)=1$. 
Then ${F}_4\wedge ({F}_1-{F}_2)$ is constant. Consequently, $\phi^2_1$ vanishes identically.  
This implies that $\phi^1_0$ and $\phi^2_0$ are closed. Let $\tilde{\zeta}$ be a nowhere zero 1-form on $\Lambda$. 
Since ${F}_3=\Gamma\circ \pi_{\Lambda}$, then $\phi^1_0+\phi^2_0=\rho  \pi_{\Lambda}^*(\tilde{\zeta})$, where $\rho$ 
is a nowhere vanishing smooth function. On the other hand, the 1-form $\phi^1_0+\phi^2_0$ is closed. This implies that 
${d}\rho\wedge \pi_{\Lambda}^*(\tilde{\zeta})=0$. Then, $\rho$ is constant on the fibers of $\pi_{\Lambda}$. 
Therefore, there exists a nowhere zero smooth function $\varrho:\Lambda\to \R$, such that
$\rho=\varrho\circ \pi_{\Lambda}$.  Put $\zeta = \varrho \tilde{\zeta}$, then $\pi_{\Lambda}^*(\zeta)=\phi^1_0+\phi^2_0$. 
Since $\phi^2_1=0$, the 1-form $\alpha$ is closed. Then, ${dq}\wedge \zeta=0$. This in turn implies that 
$q=h\circ \pi_{\Lambda}$, 
where $h:\Lambda\to \R$ is a smooth function. Differentiating ${F}_3=\Gamma\circ \pi_{\Lambda}$ yields
\begin{equation}\label{s3.2.1}
  {F}_1-{F}_2=\Gamma'\circ \pi_{\Lambda},\quad F_4=-\varepsilon \Gamma''\circ \pi_{\Lambda},
   \quad {d}({F}_3\wedge ({F}_1-{F}_2)\wedge {F}_4)=0.
      \end{equation}
This shows that $\Gamma$ is a biisotropic curve, such that ${\mathbb V}_{\Gamma}={\mathbb V}_{f}$. 
In addition, $\zeta$ is the affine line element of $\Gamma$ relative to the area element ${dA}$. Next, we choose 
a lightcone basis $\mathfrak{C}$ calibrated to $\Gamma$ 
and $\zeta$ and consider the associated isotropic normal vector field ${\rm R}$. By construction, ${\rm R}$ satisfies 
the identities \eqref{s3.1.5}. 
We write ${\rm R}\circ \pi_{\Lambda}=\sum_{j=0}^4 r_j {F}_j$, where $r_0,\dots,r_4$ are smooth functions. 
From \eqref{s3.2.1} and \eqref{s3.1.5} it follows that $r_0=\varepsilon$, $r_1+r_2=r_3=r_4=0$. 
Then ${F}_0=\varepsilon {\rm R}\circ \pi_{\Lambda}-r_1\Gamma'\circ \pi_{\Lambda}$. 
{This implies that $f({ x})=[{F}_0({ x})]\in f_\Gamma^{\pm}({\mathbb T}^{\pm}_\Gamma)$}, 
for every ${x}\in X$.
\end{proof}

\begin{figure}[ht]
\begin{center}
\includegraphics[height=6cm,width=6cm]{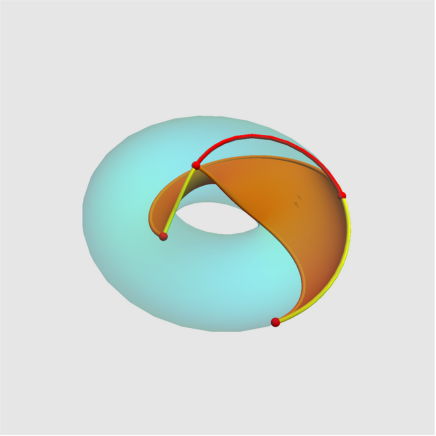}
\includegraphics[height=6cm,width=6cm]{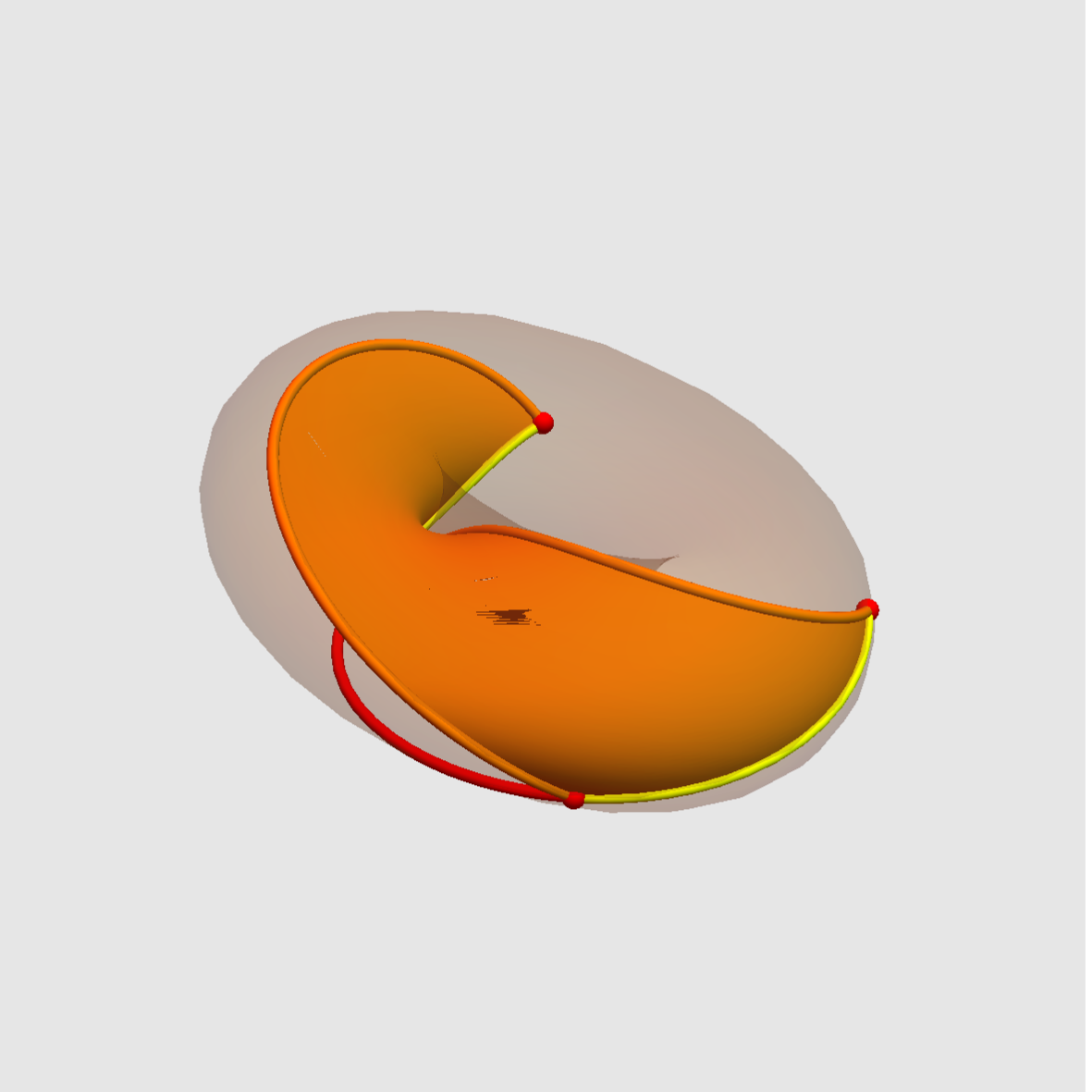}
\caption{\small The positive (on the left) and the negative (on the right) parabolic components of the tube 
along the ``biisotropic equilateral hyperbola" \eqref{bii-hyperbola}. The yellow arcs  are the umbilical curves, while the red arcs 
are the images of the dual maps.
}\label{FIG6}
\end{center}
\end{figure}

\begin{remark}
Theorem \ref{thmB} shows that, locally, any exceptional quasi-umbilical immersion $f$ is a parabolic component of the 
normal tube of a biisotropic curve. 
The rank of the dual map $f^{\sharp}$ is less than or equal to 1 and $df^{\sharp}|_{{ x}}=0$ if and only 
if the affine curvature of the directrix curve vanishes at $\pi_{\Lambda}({ x})$.
\end{remark}

\begin{prop}\label{pr1.s3.2}
Let $\Gamma : \Lambda\to \SS$ and $\tilde{\Gamma} : \tilde{\Lambda}\to \SS$ be two biisotropic curves defined on open intervals. 
Let $\zeta$ and $\tilde{\zeta}$ be affine line elements of $\Gamma$ and $\tilde{\Gamma}$, respectively. 
If the affine lengths $\ell=\int_{\Lambda}\zeta$ and $\tilde{\ell}=\int_{ \tilde{\Lambda}}\tilde{\zeta}$ are finite, 
then $f^{\pm}_{\Gamma}$  and $f^{\pm}_{\tilde{\Gamma} }$ are second order deformations of each other (with unfixed parameters). 
\end{prop}

\begin{proof}
The affine line elements are defined up to a constant positive multiplicative factor. So, we may assume that $\ell= \tilde{\ell}=2$.
Let  $s$ and $\tilde{s}$ be smooth primitives  of $\zeta$ and $\tilde{\zeta}$ such that $s(\Lambda)=\tilde{s}(\tilde{\Lambda})=(-1,1)$. 
By possibly replacing $\Gamma$ with $\Gamma\circ s^{-1}$ and $\tilde{\Gamma}$ with $\tilde{\Gamma}\circ \tilde{s}^{-1}$,
we may assume $\Lambda= \tilde{\Lambda}=(-1,1)$ and $\zeta = \tilde{\zeta}=ds$.  Let  $\mathfrak{C}$ and $\tilde{\mathfrak{C}}$ 
be lightcone bases calibrated 
to $(\Gamma,\zeta)$ and  $(\tilde{\Gamma},\tilde{\zeta})$, respectively.  
Denote  by ${\rm R}$ and $\tilde{{\rm R}}$ the corresponding isotropic normal vectors fields, defined as in \eqref{s3.1.5}, 
and let $\gamma(s) =(x(s),y(s))$ and 
$\tilde{\gamma}(s)=(\tilde{x}(s),\tilde{x}(s))$ be the affine reductions of $\Gamma$ and $\tilde{\Gamma}$ 
relative to $\mathfrak{C}$ and $\tilde{\mathfrak{C}}$. 
Consider the adapted frames ${F}_{\pm}$ and  $\tilde{{F}}_{\pm}$ 
along $f^{\pm}_{\Gamma}$ and 
$ f^{\pm}_{\tilde{\Gamma}}$ relative to $\mathfrak{C}$ and $\tilde{\mathfrak{C}}$ . 
Let $\Phi$ and $\tilde{\Phi}$ be the smooth diffeomorphisms
\[
  \begin{split}
   \Phi: (-1,1)\times \R \ni (s,u)\mapsto \Big(s, u-\int_{0}^{s}\frac{\mu^2+(xx'-yy')^2}{\upsilon^4}dt\Big)\in (0,1)\times \R,\\
     \tilde{\Phi}: (-1,1)\times \R\ni (s,u)\mapsto \Big(s, u -\int_{0}^{s}\frac{\tilde{\mu}^2+(\tilde{x} \tilde{x}'- \tilde{y} \tilde{y}')^2}{\upsilon^4}dt\Big)\in (0,1)\times \R.
   \end{split}
    \]
Then, ${G}={F}_{\pm}\circ \Phi$ and $\tilde{{G}}=\tilde{{F}}_{\pm}\circ \tilde{\Phi}$ are adapted frames along 
$f^{\pm}_{\Gamma}\circ \Phi$ and $f^{\pm}_{\tilde{\Gamma}}\circ \tilde{\Phi}$, respectively. 
From \eqref{s3.2.0.3}
it follows that the 1-forms $\phi= {G}^{-1}{dG}$ and $\tilde{\phi}= \tilde{{G}}^{-1}{d\tilde{G}}$ are is as in 
\eqref{s2.2.3}, with 
$$
  \begin{cases}\phi^2_1=\tilde{\phi}^2_1=0,\\
   \phi^1_0=\tilde{\phi}^1_0=\frac{1}{2}(du-ds),\\  
    \phi^2_0=\tilde{\phi}^2_0=-\frac{1}{2}(du+ds),
     \end{cases}
     $$
and $r=\tilde{r}=0$, $ \alpha=hds$, $ \tilde{\alpha}=\tilde{h}ds$, where $h$ and $\tilde{h}$ are the affine curvatures of $\Gamma$ 
and $\tilde{\Gamma}$. Let $\eta$ be the closed 1-form defined as in \eqref{s2.2.6}, with $\ell=\tilde{h}-h$. 
Then, $\delta=Ad_{G}( \eta)$ is an infinitesimal second order deformation of  $f^{\pm}_{\Gamma}\circ \Phi$. 
It is easily seen that $\delta=\Delta^{-1}\cdot d\Delta$, where $\Delta = \tilde{G}{G}^{-1}$.  By construction, 
  $f^{\pm}_{\tilde{\Gamma}}\circ \tilde{\Phi}=\Delta\cdot f^{\pm}_{\Gamma}\circ \Phi$. 
  This implies that $f^{\pm}_{\tilde{\Gamma}}$ is a second order deformation of $f^{\pm}_{\Gamma}$, 
  that $\Delta\circ  \tilde{\Phi}^{-1}$ is a map realizing the deformation, and that $\Phi\circ   \tilde{\Phi}^{-1}$  
  is the change of parameters.
 \end{proof}
 
 \begin{figure}[ht]
\begin{center}
\includegraphics[height=6cm,width=6cm]{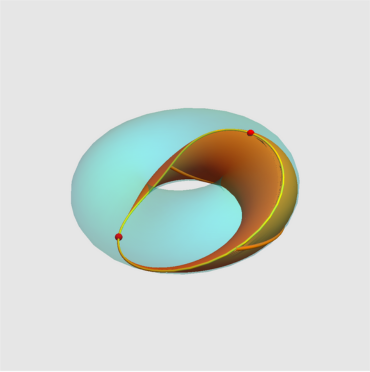}
\includegraphics[height=6cm,width=6cm]{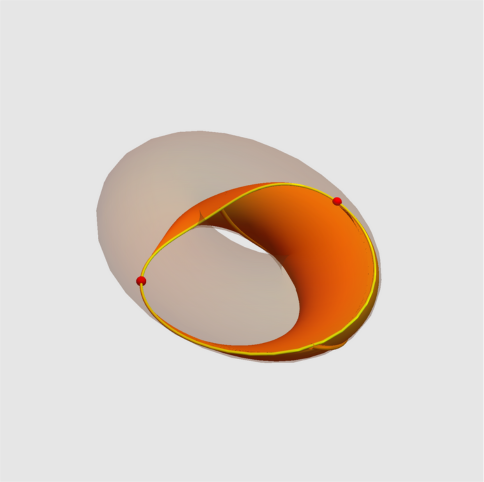}
\caption{\small The positive (on the left) and the negative (on the right) parabolic components of the tube 
along the \``biisotropic parabola" \eqref{bii-parabola}. The yellow arcs are the umbilical curves. The red conical points of the 
parabolic components are the images of the dual maps.
}\label{FIG7}
\end{center}
\end{figure}

\subsection{The normal tubes of conics and homogeneous quasi-umbilical immersions }\label{s3.3}  

An affine plane curve with constant affine curvature $h$ is a conic. 
By our convention (cf. Remark \ref{r:curv-sign}), for constant negative values of $h$ the curve 
is an ellipse, for constant positive values of $h$ it is a connected component of a hyperbola, 
and for $h=0$ it is a parabola. As for the construction of the normal 
tubes it suffices to consider the cases $h=-1$, $h=1$, and $h=0$. 

\subsubsection{The normal tube of a unit circle} 

Consider the ``biisotropic circle"
\begin{equation}\label{bii-circle}
  \Gamma:  \R/2\pi\Z \ni s \mapsto \cos(s)({E}_1-{E}_2)+{E}_3+\sin(s) {E}_4\in \SS.
    \end{equation}
Then
$$
    {\mathrm{R}}(s)=\sin(s){E}_0+\cos(s){E}_1+{E}_3+\frac{1}{2}\sin(s){E}_4
       $$
is an isotropic normal vector field along $\Gamma$. The timelike immersed torus originated by 
$\Gamma$ is
$$
    f:{\mathbb T}=\R^2/2\pi\Z^2 \ni (s,\theta)\mapsto [\cos(\theta)\rm{R}(s)+ {\frac12}\sin(\theta)\Gamma'(s)]\in \EE.
         $$
Each one of the umbilical curves ${\rm C}_{+}$, ${\rm C}_{-}$ is taken by $f$ onto 
the null geodesic ${\mathcal C}=\{[ \sum_{j=0}^{4}x_j {E}_j]\in \EE \mid {x^2=x^3=0}\}$ 
lying in the AdS wall $\partial {\mathcal A}^{\pm}_{{ E}_3}$  (see Figure \ref{FIG4}). 
The parabolic components ${\mathbb T}^{\pm}$ are two cylinders and the exceptional quasi-umbilical immersions  $f^{\pm}_{\Gamma}$ 
are embeddings. The parabolic component $f^{+}_{\Gamma}({\mathbb T}^{+}_{\Gamma})$ is contained in the AdS 
chamber ${\mathcal A}^{+}_{{ E}_3}$ and $f^{-}_{\Gamma}({\mathbb T}^{-}_{\Gamma})$ is contained in ${\mathcal A}^{-}_{{ E}_3}$. 
The central torus $\partial \mathcal{A}|_{(s_*,\theta_*)}$ is the AdS wall of $\Gamma|_s$.  It is tangent to $f( {\mathbb T})$ 
along the coordinate line $s=s_*$. The image of the dual maps $f^{\sharp}_{\pm}$ is the umbilical curve  ${\mathcal C}$ 
and  all central tori pass through ${\mathcal C}$ (see Figure \ref{FIG5}).

\subsubsection{The normal tube of a equilateral hyperbola} 

Consider the ``biisotropic equilateral hyperbola"
\begin{equation}\label{bii-hyperbola}
 \Gamma:  \R \ni s \mapsto \sinh(s)({E}_1-{E}_2)+{E}_3+\cosh(s){E}_4\in \SS.
     \end{equation}
Then
$$
     {\rm{R}}(s)=-\cosh(s){E}_0+\sinh(s){E}_1-{E}_3-\frac{1}{2}{E}_4
       $$
is an isotropic normal vector field along $\Gamma$.  So,  $\Gamma$ originates the timelike embedded cylinder
$$
     f: {\mathbb T}=\R\times \R/2\pi \Z \ni (s,\theta) \mapsto [\cos(\theta)\rm{R}(s)+ {\frac12}\sin(\theta)\Gamma'(s)]\in \EE.
          $$
The umbilical curves ${\rm C}_{+}$ and ${\rm C}_{-}$ are mapped onto two disjoint arcs ${\mathcal C}_{\pm}$ 
of the null geodesic ${\mathcal C}$, the umbilical arcs of $f$. 
The parabolic component $f^{+}_{\Gamma}({\mathbb T}^{+})$ is contained in the negative AdS chamber 
${\mathcal A}^{-}$, while the other one is contained in the positive AdS chamber.  
The images of the dual maps $f^{\sharp}_{\pm}$ are two disjoint open arcs  ${\mathcal C}^{\sharp}_{\pm}$ contained in 
the null geodesic ${\mathcal C}$. They do not intersect the umbilical arcs. Taking the limits of $f(s,\theta)$ as $s\to \pm\infty$,
 we get two disjoint null geodesics ${\mathcal B}^{+}$ and ${\mathcal B}^{-}$ of the AdS wall $\partial {\mathcal A}_{{E}_3}$. 
 They intersect ${\mathcal C}$ into four distinct points. The null geodesic ${\mathcal C}$ is the disjoint union of these four 
 points with the arcs ${\mathcal C}_{\pm}$ and ${\mathcal C}^{\sharp}_{\pm}$. 
 As in the previous case, all central tori of $f$ pass through the null geodesic ${\mathcal C}$ 
(see Figure \ref{FIG6}).

\subsubsection{The normal tube of a parabola} 

Consider the ``biisotropic parabola''
\begin{equation}\label{bii-parabola}
    \Gamma:  \R\ni s  \mapsto s({E}_1-{E}_2)+{E}_3+\frac{s^2}{2}{E}_4\in \SS.
       \end{equation}
The map
$$
  {\mathrm{R}}(s)=-{E}_0+\frac{s(2+s^2)^2}{8(1+s^2)}{E}_1+\frac{s}{8}(4-\frac{s^4}{1+s^2}){E}_2
   +\frac{s^2}{2}{E}_3-\frac{s^4}{8(1+s^2)}{E}_4
      $$
is an isotropic normal vector field along $\Gamma$ and the timelike tautological immersion of the normal tube can be written as
$$
      f: \R\times \R/2\pi\Z \ni (s,\theta)\mapsto [\cos(\theta)\rm{R}(s)+ {\frac12}\sin(\theta)\Gamma'(s)]\in \EE.
         $$
The images of the umbilical curves $f({\rm C}_{+})$ and $f({\rm C}_{-})$ are two adjacent disjoint open arcs of ${\mathcal C}$. Their common 
boundary consits of the two points $[\pm {E}_4]$. The parabolic components $f^{\pm}_{\Gamma}({\mathbb T}^{\pm})$ are 
contained in the AdS chambers ${\mathcal A}^{\pm}_{{E}_3}$ and the dual maps are constants. Also in this case 
the central tori  pass through the null geodesic ${\mathcal C}$ (see Figure \ref{FIG7}).

\subsubsection{Homogeneous quasi-umbilical immersions} 

Let $f:X\to \EE$ be a timelike immersed surface. A  \textit{conformal symmetry} of $f$ is a pair $(\Phi,{B})$ consisting 
of a diffeomorphism $\Phi:X\to X$ and an element ${B}\in \A$, such that $B\cdot (f\circ \Phi)=f$. 
The totality of conformal symmetries of $f$ is a group $G_f$, the group of conformal symmetries of $f$, with 
multiplication given by $(\Phi,{B})\star (\Phi',{B}')=(\Phi'\circ\Phi,{B}\, {B}')$. 
The group $G_f$ acts on $X$
on the left by $(\Phi,{B}){ x}=\Phi({ x})$, for all $(\Phi,{B})\in G_f$ and for all $x\in X$.  
If this action is transitive, the immersion is said to be \textit{homogeneous}.  
For instance, the timelike totally umbilical tori of $\EE$ are homogeneous and have a six-dimensional group of symmetries. 
Other examples are the parabolic components of the tubes of a biisotropic conic of $\SS$. In these cases, 
the symmetry group is isomorphic to a 2-dimensional abelian subgroup of $\A$. 
In \cite{THE}, The showed that a generic quasi-umbilical immersion cannot be homogeneous. It is not difficult to show that 	
\textit{an exceptional quasi-umbilical immersion is homogeneous if and only if it is $\A$-equivalent to a parabolic component of the 
normal tube of a biisotropic conic}.

\section{Quasi-umbilical immersions: the general case}\label{s4}

\subsection{Generic null curves in the neutral space form $\SS$}\label{s4.1} 

\begin{defn}
An immersed null curve $\Gamma: \Lambda\to \SS$ is said to be \textit{generic} if $\langle \Gamma'',\Gamma'' \rangle$ 
is either strictly positive or strictly negative. 
A \textit{line element} is a nowhere zero 1-form $\zeta \in \Omega^1(\Lambda)$, {such that}
$$
     \langle \Gamma'',\Gamma''\rangle=-2\eta, \quad \eta= {-}\sgn \langle \Gamma'',\Gamma'' \rangle=\pm 1.
        $$ 
The line element is unique up to the sign.
\end{defn}

\subsubsection{The left and right normal tubes}

Let $\zeta$ be a line element. 
Differentiating the identities $\langle \Gamma,\Gamma\rangle =1$, $\langle \Gamma', \Gamma'\rangle = 0$ and 
$\langle \Gamma'',\Gamma''\rangle = -2\eta$,  it follows that ${\mathcal R}_1|_{\tau}:=\span\{\Gamma'|_{\tau},\Gamma'''|_{\tau}\}$ 
is a plane of type $(1,1)$ and ${\mathcal R}_2|_{\tau}:=\span\{\Gamma|_{\tau},\Gamma'|_{\tau},\Gamma'''|_{\tau}\}$ is a subspace 
of type $(1,2)$, for every $\tau \in \Lambda$. 
The polar space ${\mathcal R}^{\perp}_2|_{\tau}$ of ${\mathcal R}_2|_{\tau}$ is a  plane of type $(1,1)$ 
such that $\Gamma''|_{\tau}\in {\mathcal R}^{\perp}_2|_{\tau}$, for every $\tau$. 
The polar space ${\mathcal R}^{\perp}_2|_{\tau}$ has a natural orientation: 
a basis $({V},{W})$ of  ${\mathcal R}^{\perp}_2|_{\tau}$ is positive if and only if $-\eta {V}\wedge \Gamma|_{\tau}\wedge\Gamma'|_{\tau}\wedge \Gamma'''|_{\tau}\wedge {W}>0$. 
Let ${\mathcal R}_1$, ${\mathcal R}_2$, and ${\mathcal R}^{\perp}_2$ be  the  vector bundles
\begin{equation}\label{s4.1.a1}\begin{cases}
 \pi_1: {\mathcal R}_1=\{(\tau,{ V})\in \Lambda\times \MM \mid { V}\in {\mathcal R}_1|_{\tau}\}\to \Lambda,\\
   \pi_2: {\mathcal R}_2=\{(\tau,{ V})\in \Lambda\times \MM  \mid { V}\in {\mathcal R}_2|_{\tau}\}\to \Lambda,\\
     \pi^{\perp}_2: {\mathcal R}^{\perp}_2=\{(\tau,{V})\in \Lambda\times \MM  \mid {V}\in {\mathcal R}_2^{\perp}|_{\tau}\}\to \Lambda
       \end{cases}
         \end{equation}
 The \textit{dual} of a null vector ${V}\in {\mathcal R}^{\perp}_2|_{\tau}$ is a null vector ${V}_*\in {\mathcal R}^{\perp}_2|_{\tau}$,
 such that $\langle {V},{V}_*\rangle = -1$. We say that a null vector ${V}\in {\mathcal R}^{\perp}_2|_{\tau}$  is \textit{right-handed} 
  if $({V},{V}_*)$ is a positive basis of ${\mathcal R}^{\perp}_2|_{\tau}$. 
 Right-handed null vectors generate a null line subbundle 
${\mathcal P}$ of ${\mathcal R}^{\perp}_2$. The function  ${\mathcal P}\setminus\{0\}\ni {V}\mapsto \langle {V},\Gamma'' \rangle$ 
is nowhere vanishing. Then there exists a unique trivialization  ${\rm C}_{\lambda}$ of ${\mathcal P}$, such that $\langle {\rm C}_{\lambda},\Gamma''\rangle = 1$. We denote by ${\rm C}_{\varrho}$ the unique dual of ${\rm C}_{\lambda}$, such that $\langle {\rm C}_{\varrho},\Gamma''\rangle = \eta$. By construction, $({\rm C}_{\lambda},{\rm C}_{\varrho})$ is a positive null trivialization of 
${\mathcal R}^{\perp}_2$, such that
\begin{equation}\label{s4.1.1}
 -\eta {\rm C}_{\lambda}\wedge \Gamma\wedge \Gamma'\wedge \Gamma'''\wedge {\rm C}_{\varrho}>0,\quad
   \Gamma''=-(\eta {\rm C}_{\lambda}+ {\rm C}_{\varrho}). 
     \end{equation}

\begin{defn}
Let $\zeta$ be the unique line element, such that the null bivectors ${\rm C}_{\varrho}|_\tau\wedge \Gamma'|_\tau$ are positive, 
for every $\tau\in \Lambda$. 
We call $\zeta$ the \textit{canonical line element} of $\Gamma$. The \textit{left and right isotropic normal tubes} along $\Gamma$ 
are the trivial circle bundles
$$
 \begin{cases}
 \pi_{\lambda}:{\mathbb T}_\lambda=\{(\tau,[{V}])\in \Lambda\times \EE \mid {V}\in [({\rm C}_{\lambda}\wedge \Gamma')|_{\tau}]
\}\to \Lambda,\\
\pi_{\varrho}:{\mathbb T}_{\varrho}=\{(\tau,[{V}])\in \Lambda\times \EE \mid {V}\in  [({\rm C}_{\varrho}\wedge \Gamma')|_{\tau}]
\}\to \Lambda.
     \end{cases}
           $$
           \end{defn}
The elements of 
${\mathbb T}_\lambda$ can be written as $(\tau,[{V}_\lambda(\tau,\theta)])$, where 
\begin{equation}\label{s4.1.2}
       {V}_\lambda(\tau,\theta)=\cos(\theta){\rm C}_{\lambda}|_\tau+\sin(\theta)\Gamma'|_{\tau},
           \end{equation}
while the elements of ${\mathbb T}_{\varrho}$ are of the form $(\tau,[{V}_{\varrho}(\tau,\theta)])$, where
\begin{equation}\label{s4.1.3}
   {V}_{\varrho}(\tau,\theta)=\cos(\theta){\rm C}_{\varrho}|_\tau+\sin(\theta)\Gamma'|_{\tau}.
   \end{equation}
%
The normal tubes do intersect along the curves 
$$
  {\rm C}_{\pm}=\{(\tau, [\pm \Gamma'|_{\tau}]),\tau \in \Lambda\},
     $$
     called \textit{umbilical lines}.
The complementary set of ${\rm C}_{+}\cup {\rm C}_{-}$ has two connected components,
$$
    {\mathbb T}_\lambda^{\pm}=\{(\tau,[{V}_\lambda(\tau,\theta)]) \mid {\sgn}(\cos \theta)=\pm 1\}
      $$
and 
$$
   {\mathbb T}_{\varrho}^{\pm}=\{(\tau,[{V}_{\varrho}(\tau,\theta)])\mid {\sgn}(\cos \theta)=\pm 1\}.
     $$

\begin{defn}
We call ${\mathbb T}_\lambda^{\pm}$ and ${\mathbb T}_{\varrho}^{\pm}$ the \textit{parabolic components} 
of the normal tubes. The map
$$
J_{\lambda}^{\varrho}:  {\mathbb T}_\lambda\ni (\tau,[{V}_\lambda(\tau,\theta)]) \mapsto (\tau,[{V}_{\varrho}(\tau,\theta)])\in {\mathbb T}_{\varrho},
    $$
is the \textit{intertwining diffeomorphism} between the left and right normal tubes. The inverse 
diffeomorphism $(J_{\lambda}^{\varrho})^{-1}$ 
will be denoted by 
$J_{\varrho}^{\lambda}$.
\end{defn}

\begin{defn}
The left and right normal tubes can be mapped into the Einstein universe by means of the \textit{left and right tautological maps}
$$
   f_{\lambda}: {\mathbb T}_\lambda\ni (\tau,[{V}]) \mapsto [{V}]\in \EE,     \quad 
     f_{\varrho}: {\mathbb T}_{\varrho}\ni (\tau,[{V}]) \mapsto [{V}]\in \EE.
         $$
The restrictions of the tautological maps to the parabolic components are denoted by $f^{\pm}_{\lambda}$ and by 
$f^{\pm}_{\varrho}$, respectively.
\end{defn}

 \begin{figure}[ht]
\begin{center}
\includegraphics[height=6cm,width=6cm]{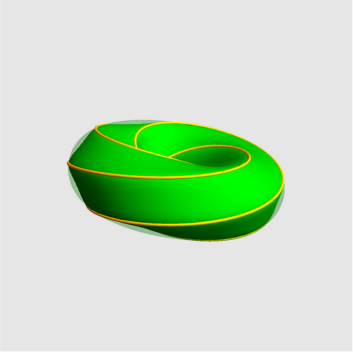}
\includegraphics[height=6cm,width=6cm]{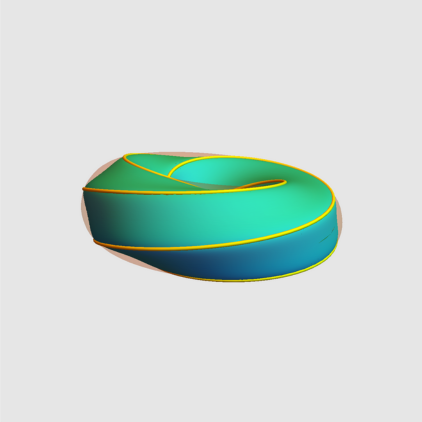}
\caption{\small The left normal tube along a generic null 
curve with constant curvatures ($m=3$ 
and $n=5$). On the left, the positive parabolic component; on the right, the negative parabolic component. 
The yellow line is the umbilical curve of the surface. The umbilical points are double points. The parabolic 
components are embedded cylinders.
}\label{FIG8}
\end{center}
\end{figure}

 \begin{figure}[ht]
\begin{center}
\includegraphics[height=6cm,width=6cm]{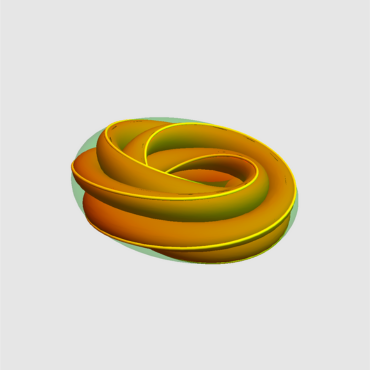}
\includegraphics[height=6cm,width=6cm]{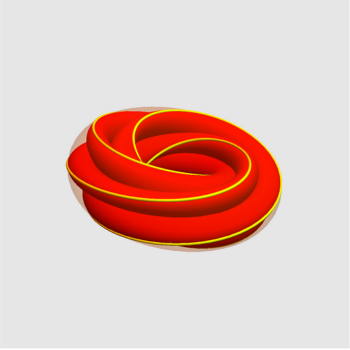}
\caption{\small The right normal tube along a generic null 
curve with constant curvatures 
($m=3$ and $n=5$). On the left, the positive parabolic component; on the right, the 
negative parabolic component. The yellow line is the umbilical curve of the surface. In this case 
the parabolic components are immersed but not embedded cylinders.
}\label{FIG9}
\end{center}
\end{figure}

\subsection{Structure of quasi-umbilical immersions of general type}\label{s4.2}  

With reference to the constructions on generic null curves in $\SS$ developed in Section \ref{s4.1}, we 
can now describe the geometric structure of general quasi-umbilical immersions.
The following is the main result of this section.

\begin{thmx}\label{thmC}
Let $\Gamma : \Lambda \to \SS$ be a generic null immersion. 
\vskip0.1cm
 $(1)$ The tautological map $f_{\lambda}$ is a timelike immersion with 
umbilic locus ${\rm C}_{+}\cup {\rm C}_{-}$ and conformal Gauss map ${\mathcal N}_{f_\lambda}=\Gamma\circ \pi_{\lambda}$. 
The restrictions $f^{\pm}_{\lambda}$ of $f_{\lambda}$ to the parabolic components ${\mathbb T}_\lambda^{\pm}$ 
of the left normal tube ${\mathbb T}_\lambda$
are quasi-umbilical immersions 
of general type with helicity $\varepsilon = \pm1$ and 
dual maps $(f^{\pm}_{\lambda})^{\sharp}=f^{\pm}_{\varrho}\circ J^{\varrho}_{\lambda}$. 
\vskip0.1cm
 $(2)$ The tautological map $f_{\varrho}$ is a timelike immersion with 
umbilic locus ${\rm C}_{+}\cup {\rm C}_{-}$ and conformal Gauss map ${\mathcal N}_{f_\varrho}=\Gamma\circ \pi_{\varrho}$. 
The restrictions  $f^{\pm}_{\varrho}$ of $f_{\varrho}$ to the parabolic components ${\mathbb T}_\varrho^{\pm}$ 
of the right normal tube ${\mathbb T}_\varrho$ are quasi-umbilical immersions of general 
type with helicity $\varepsilon = \pm1$ and 
dual maps $(f^{\pm}_{\varrho})^{\sharp}=f^{\pm}_{\lambda}\circ J_{\varrho}^{\lambda}$. 

Conversely, let $f:X\to \EE$ be a regular quasi-umbilical 
immersion of general type. Then $f(X)$ is contained in one of the two parabolic components 
$f^{\pm}_{\lambda}({\mathbb T}_{\lambda}^{\pm})$ of the tautological immersion originated by the directrix curve of $f$.
\end{thmx}

\begin{proof}
Let $\Gamma:\Lambda\to \SS$ be a generic null immersion, $\zeta$ be its canonical line element, and let ${B}:\Lambda\to \A$ 
be the map defined by
$$
   {B}_0={\rm C}_{\lambda},\quad { B}_1-{ B}_2=\Gamma',\quad  2{ B}_1\wedge  { B}_2
    =-\eta\Gamma'\wedge \Gamma''',\quad  {B_3 = \Gamma}, \quad
       {B}_4={\rm C}_{\varrho}.
         $$
First, we claim that ${B}$ is a solution of the linear system
\begin{equation}\label{fls}{B}^{-1}{B}'= \left(
           \begin{array}{ccccc}
             0&-\frac{\eta}{2} +\kappa_{\varrho} & \frac{\eta}{2} +\kappa_{\varrho} &0&0 \\
             \frac{1}{2} -\kappa_{\lambda}& 0 & 0&1& \frac{\eta}{2} -\kappa_{\varrho} \\
                \frac{1}{2} +\kappa_{\lambda}& 0 & 0&-1&  \frac{\eta}{2} +\kappa_{\varrho} \\
                0&1&1&0&0\\
             0 &- \frac{1}{2} +\kappa_{\lambda} &  \frac{1}{2} +\kappa_{\lambda}&0&0 \\
           \end{array}
         \right),
              \end{equation}
where $\kappa_{\lambda}$ and $\kappa_{\varrho}$ are smooth functions, the \textit{left and right curvatures} 
of $\Gamma$.
Denote by $b^i_j$, $0\le i, j\le 4$, the components of the $\aa$-valued map ${ B}^{-1}{B}'$. 
Differentiating ${B}_3=\Gamma$, recalling that $\Gamma'={B}_1-{B}_2$, we have
\begin{equation}\label{l1.s4.1.1}
      b^1_3 =-b^2_3 =1,\quad b^0_3 = b^4_3=0,
        \end{equation}
and hence $b^3_1 = b^3_2 =1$, ${b^3_4} = b^3_0=0$, 
by the symmetry relations between the $b^i_j$.
Observe that, since ${B}_0={\rm C}_{\lambda}$ and ${B}_4={\rm C}_{\varrho}$ 
are orthogonal to $\Gamma$ and $\Gamma'$, differentiating $\langle {B}_0,\Gamma\rangle = \langle {B}_4,\Gamma\rangle=0$, 
we obtain $\langle {B}'_0,\Gamma\rangle = \langle {B}'_4,\Gamma\rangle=0$, i.e., $b^3_0=b^3_4=0$.
Differentiating ${B}_1-{B}_2=\Gamma'$, taking into account that
$$
        \Gamma''=-(\eta{\rm C}_{\lambda}+{\rm C}_{\varrho})=-(\eta{B}_0  + {B}_4),
          $$
yields
\begin{equation}\label{l1.s4.1.3}
       b^2_1=0,\quad b^1_4+b_4^2=\eta,\quad b^1_0+b^2_0=1.
           \end{equation}
This implies the existence of two smooth functions $\kappa_{\lambda}$ and $\kappa_{\varrho}$, such that
\begin{equation}\label{l1.s4.1.4}
    b^1_0=\frac{1}{2}-\kappa_{\lambda},\quad b^2_0=\frac{1}{2}+\kappa_{\lambda},\quad 
     b^1_4=\frac{\eta}{2}-\kappa_{\varrho},\quad b^2_4=\frac{\eta}{2}+\kappa_{\varrho}.
           \end{equation}
The third derivative of $\Gamma$ is a cross section of the bundle spanned by ${B}_1$ and ${B}_2$. Then, 
differentiating $\langle \Gamma'',{\rm C}_{\lambda}\rangle = 1$, we find $\langle \Gamma'',{B}'_0\rangle=0$,
and hence, $b^0_0=0$. This together with
\eqref{l1.s4.1.1}, 
\eqref{l1.s4.1.3} and \eqref{l1.s4.1.4} proves the claim.
\vskip0.1cm
Next, we identify the normal tubes with $\Lambda\times \R/2\pi\Z$ as in 
\eqref{s4.1.2} and \eqref{s4.1.3}. Given a map $\Psi: \Lambda\times \R/2\pi\Z\to \R^k$, we set
$d\Psi=\partial_{\zeta}\Psi\zeta+\partial_{\theta}\Psi d\theta$. Let us consider the lift of $f_{\lambda}$ defined by
${\rm T}_{\lambda}=(\cos\theta){B}_0+(\sin\theta)({B}_1-{B}_2)$. From \eqref{fls}, we compute
$$
\langle d{\mathrm{T}}_{\lambda},d{\mathrm{T}}_{\lambda}\rangle 
   =-2\zeta\big(d\theta+(\eta \sin^2\theta - \kappa_{\lambda} \cos^2\theta)\zeta\big).
     $$
This shows that  $f_{\lambda}$ is a timelike immersion. Let ${F}^{\pm}:{\mathbb T}_{\lambda}^{\pm}\to \A$ 
be the frame fields along $f_{\lambda}$ defined by
$$
\begin{cases}
{F}^{\pm}_0=\pm\big({ B}_0+ (\tan\theta)({ B}_1-{ B}_2)\big),\\
{F}^{\pm}_1=\pm\big((\tan\theta)(\frac{\eta}{3}{ B}_0-{ B}_4)+{ B}_1+\frac{\eta}{3}(\tan\theta)^2({ B}_1-{ B}_2)\big),\\
{F}^{\pm}_2=\pm\big((\tan\theta)(\frac{\eta}{3}{ B}_0-{ B}_4)+{ B}_2+\frac{\eta}{3}(\tan\theta)^2({ B}_1-{ B}_2)\big),\\
{F}^{\pm}_3= B_3 = \Gamma,\\
{F}^{\pm}_4=
{\pm \big(B_4 - \frac{\eta}{3}(\tan\theta)({ B}_1-{ B}_2)\big).}
   \end{cases}
    $$
Using \eqref{fls} one computes the $\aa$-valued 1-forms $({ F}^{\pm})^{-1}{d F}^{\pm}$, which yield
\begin{equation}\label{s4.2.1}
\begin{cases}
\phi_0^3=\phi^0_3=\phi^4_3=\phi^3_4=0,\\ 
\phi^0_0=2\phi^2_1=-\frac{4\eta}{3}(\tan\theta)\zeta,\quad \phi_1^3=\phi_2^3=\pm \zeta\\
\phi^1_0=(\frac{1}{2}-\kappa_{\lambda}+\eta(\tan\theta)^2)\zeta +(\sec\theta)^2d\theta,\\
\phi^2_0=(\frac{1}{2}+\kappa_{\lambda}-\eta(\tan\theta)^2)\zeta -(\sec\theta)^2d\theta,\\
\phi^1_4=\frac{\eta}{18}(9-18\eta\kappa_{\varrho}+2 {\eta}(\tan\theta)^2)\zeta-\frac{\eta}{3}(\sec\theta)^2d\theta,\\
\phi^2_4=\frac{\eta}{18}(9+18\eta\kappa_{\varrho}-2 {\eta}(\tan\theta)^2)\zeta+\frac{\eta}{3}(\sec\theta)^2d\theta.
\end{cases}
   \end{equation}
Then, ${ F}^{\pm}$ are adapted frame fields along $f^{\pm}_{\lambda}$. This implies 
$$
     {\mathcal N}_{f^{\pm}_{\lambda}}(\tau,\theta)=\Gamma(\theta),\quad (f^{\pm}_{\lambda})^{\sharp}(\tau,\theta)=
      [{ F}^{\pm}_4(\tau,\theta)]= f^{\pm}_{\varrho}\circ J^{\varrho}_{\lambda}(\tau,\theta).
      $$
Since $\langle d{\mathcal N}_{f^{\pm}_{\lambda}},d{ F}^{\pm}_4\rangle = {\mp\eta \zeta^2}$, we infer that  $f^{\pm}_{\lambda}$ 
are of general type, with conformal Gauss map $\Gamma\circ \pi_{\Lambda}$, helicity $\varepsilon =\pm 1$ and dual map $(f^{\pm}_{\lambda})^{\sharp}=f^{\pm}_{\varrho}\circ J_{\varrho}^{\lambda}$. 
By continuity, ${\mathcal N}_{f_\lambda}=\Gamma\circ \pi_{\Lambda}$. 
From \eqref{fls}, we have
$$
     {d}{\mathcal N}_{f_\lambda}=\zeta({B}_1-{B}_2)
        $$
and
\[
\begin{split} 
 d{\rm T}_{\lambda} &=\Big(\cos(\theta)\big(\frac{1}{2}({ B}_1+{ B}_2)-\kappa_{\lambda}({ B}_1-{ B}_2)- 
  \sin(\theta)(\eta{ B}_0+{ B}_4)\big)\Big)\zeta \\
 &\qquad + \big( -\sin(\theta){ B}_0+\cos(\theta)({ B}_1-{ B}_2)\big)d\theta.
   \end{split}
    \]
Then,
$$
      \langle {d}{\mathcal N}_{f_\lambda},d{\rm T}_{\lambda}\rangle = -\cos(\theta)\zeta^2.
         $$
This proves that ${\rm C}_{+}\cup {\rm C}_{-}$ is the umbilic locus of $f_\lambda$.
{This concludes the proof of assertion (1). Similarly, one can prove assertion (2).}

\vskip0.1cm
 Conversely, let $f:X\to \EE$ be a regular quasi-umbilical immersion of general type and let $\Gamma:\Lambda\to \SS$ 
 be its directrix curve. Since $f$ is of general type, there exists a unique adapted frame field ${ F}:X\to \A$ along $f$,
  such that {$\phi^1_4+\phi^2_4=\frac{\eta}{2}(\phi^1_0+\phi^2_0)$}, $\eta = \pm1$. 
  Then, the 1-form $\phi={ F}^{-1}{d F}$ is as in \eqref{s2.2.3} 
  with $r=\eta$ and {$\tilde{\zeta}=\phi^1_0-\phi^2_0$}. Let  $\zeta$ be a nowhere zero 1-form on $\Lambda$.
  Then $\phi^1_0+\phi^2_0=\pm e^{u}\pi_{\Lambda}^*(\zeta)$, where $u:X\to \R$ is a smooth function.  
From
\begin{equation}\label{s4.2.2}
  {F}_3=\Gamma\circ \pi_{\Lambda}\quad d{F}_3=\varepsilon(\phi^1_0+\phi^2_0)({F}_1-{F}_2)
     \end{equation}
we see that $\Gamma$ is a null curve. Differentiating \eqref{s4.2.2}, using \eqref{s2.2.3}, we obtain
\begin{equation}\label{s4.2.3}
     \pi_{\Lambda}^*(\Gamma'' \zeta)=\pm \varepsilon\Big( e^{u}(d u-\phi^2_1)({F}_1-{F}_2)-
             e^{2u}(\eta { F}_0+{ F}_4) \zeta\Big).
             \end{equation}
Then, $\langle \Gamma'' \circ \pi_{\Lambda},\Gamma'' \circ \pi_{\Lambda}\rangle = {-2\eta e^{4 u}}$. 
This implies that $\Gamma$ is a generic null curve. So, we may assume that $\zeta$ is the canonical line element of $\Gamma$. 
Then, $\langle \Gamma'',\Gamma''\rangle = -2\eta$. Hence, $u =0$ and, using  \eqref{s4.2.3}, 
we infer that $\phi^2_1={\rm p}\pi_{\Lambda}^*(\zeta)$, where ${\rm p}:X\to \R$ is a smooth function. 
Hence the equation \eqref{s4.2.3} can be written as
\begin{equation}\label{s4.2.4}
  \Gamma'' \circ \pi_{\Lambda}=\mp \varepsilon\left(\mathrm{p}\,\Gamma'\circ \pi_{\Lambda}
    +(\eta { F}_0+{ F}_4)\right).
    \end{equation}
 Differentiating \eqref{s4.2.4}, we obtain
\begin{equation}\label{s4.2.5} 
  \Gamma''' \circ \pi_{\Lambda} = \mp\varepsilon\big(\mathrm{p}(\eta {F}_0\mp 3{ F}_4)
   +\frac{\eta}{2}(3{ F}_1+{ F}_2)\big).
    \end{equation}
 Then, retaining the notation of \eqref{s4.1.a1} and using \eqref{s4.2.5}, we deduce that the vector bundle 
 $\pi_{\Lambda}^*({\mathcal R}_2^{\perp})$ is spanned by the null vector fields 
  \begin{equation}\label{s4.2.6} 
    { A}={ F}_0 \pm \frac{3}{2}\eta\mathrm{p}({ F}_1-{ F}_2),\quad { B}={ F}_4 
     - \frac{1}{2}\eta{\rm p}({ F}_1-{ F}_2).
     \end{equation}
 From \eqref{s4.2.5} it follows that
   \begin{equation}\label{s4.2.7}
    -\eta { A}\wedge( \Gamma\wedge \Gamma'\wedge \Gamma''')\circ \pi_{\Lambda}\wedge {B} 
       =  2{ F}_0\wedge { F}_1\wedge { F}_2\wedge { F}_3\wedge { F}_4 >0. 
         \end{equation}
By \eqref{s4.2.6}, we have $\langle { A},{ B}\rangle = -1$ and 
$\langle { A},{ A}\rangle = \langle { B},{ B}\rangle = 0$. 
Then, \eqref{s4.2.7} implies that ${ A}$ is a global trivialization of 
$\pi_{\Lambda}^*({\mathcal P})$, 
where ${\mathcal P}$ is the line bundle generated by the right-handed null vectors of ${\mathcal R}_2^{\perp}$. 
From \eqref{s4.2.4} and \eqref{s4.2.6}, we obtain $\langle { A},\Gamma'' \circ \pi_{\Lambda}\rangle = \pm\varepsilon$. 
This implies ${ A}=\pm\varepsilon \rm{C}_{\lambda}\circ \pi_{\Lambda}$. By  \eqref{s4.2.6}, we have
 $$
     {F}_0 = \pm\varepsilon \rm{C}_{\lambda}\circ \pi_{\Lambda}\mp \frac{3}{2}\eta{\rm p}\Gamma'\circ \pi_{\Lambda}.
         $$
Then, $f(X)\subseteq f^{+}_{\lambda}({\mathbb T}^+)$ if $\pm\varepsilon=1$ and $f(X)\subseteq f^{-}_{\lambda}({\mathbb T}^-)$ 
if $\pm \varepsilon=-1$. This concludes the proof.  \end{proof}


\begin{remark}\label{r:def-gen-type}
{Let $\Gamma$, 
$\tilde{\Gamma}: \Lambda \to \SS$ 
be two generic null curves, defined on the same open interval $\Lambda$. Suppose that $\Gamma$, 
$\tilde{\Gamma}$ have the same canonical line elements. 
Applying similar arguments as those in the proof of Proposition \ref{pr1.s3.2}, 
we can prove that $f^{\pm}_{\lambda}$ and $ {f}^{\pm}_{\tilde\lambda}$ (respectively, $f^{\pm}_{\varrho}$ and ${f}^{\pm}_{\tilde\varrho}$)  
are second order conformal deformations of each other (with fixed parameters) if and only if the left curvatures (respectively, 
the right curvatures) of $\Gamma$ and $\tilde\Gamma$ coincide.}

\end{remark}

 \begin{figure}[ht]
\begin{center}
\includegraphics[height=6cm,width=6cm]{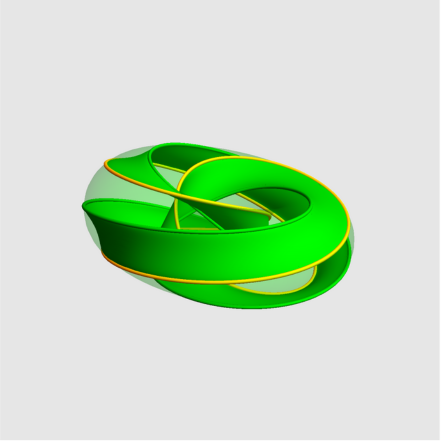}
\includegraphics[height=6cm,width=6cm]{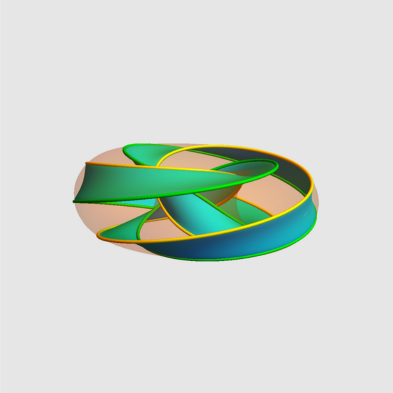}
\caption{\small The left normal tube along a generic null 
curve with constant curvatures 
($m=2$ and $n=3$). On the left, the positive parabolic component; on the right, the negative parabolic component. 
The yellow lines are the umbilical curves of the surface. In this example the tautological map 
is an embedding.
}\label{FIG10}
\end{center}
\end{figure}

 \begin{figure}[ht]
\begin{center}
\includegraphics[height=6cm,width=6cm]{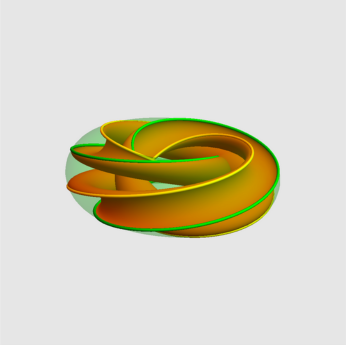}
\includegraphics[height=6cm,width=6cm]{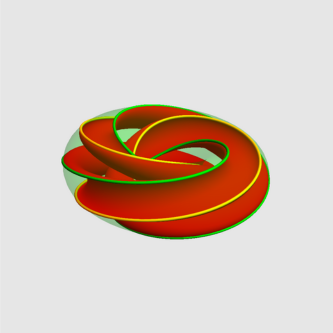}
\caption{\small The right normal tube along a generic null 
curve with constant curvatures 
($m=2$ and $n=3$). On the left, the positive parabolic component; on the right, the negative 
parabolic component. The yellow lines are the umbilical curves of the surface. 
In this example the tautological map is an embedding.
}\label{FIG11}
\end{center}
\end{figure}

\subsection{Examples}\label{s4.3} 

We illustrate our construction on a class of simple examples. The computations are carried out in the coordinates 
of $\MM$ relative to the standard pseudo-orthogonal basis $\mathfrak{P}=(P_0,\dots,P_4)$ of $\MM$. 
The details and the intermediate steps will be omitted either
because elementary or because they can be 
carried out with the help of symbolic computational systems such as 
Mathematica, Maple or Matlab. 
Let $r\in (0,1)$ and let $m, n$ be two relatively prime integers such that $n>m>0$. 
Consider
$$
 {W}(s)=
  \left(m\cos(ns),m\sin(ns),n\cos(ms),n\sin(ms),\frac{nr}{\sqrt{1-r^2}}\right).
     $$
Then 
 $$
   { \Gamma(s): = \frac{W(s)}{\sqrt{\langle {W}(s),{W}(s)\rangle}} }\in \SS
      $$
defines a generic null curve of negative type (i.e. $\langle\Gamma'',\Gamma''\rangle <0$) with constant left and right curvatures.
Thus $\Gamma$ is a homogeneous null curve of $\SS$.
The isotropic normal vector fields along $\Gamma$ take the form
\begin{eqnarray*}
  {\rm C}_{\lambda} &=&\upsilon_{\lambda}\left(c_{\lambda}^0 \cos(ns),c_{\lambda}^1\sin(ns),
         c_{\lambda}^2\cos(m s),c_{\lambda}^3\sin(m s),c_{\lambda}^4\right),
 \\
     {\rm C}_{\varrho} &=&\upsilon_{\varrho}\left(c_{\varrho}^0 \cos(ns),c_{\varrho}^1\sin(ns),
          c_{\varrho}^2\cos(m s),c_{\varrho}^3\sin(m s),c_{\varrho}^4\right),
          \end{eqnarray*}
where $\upsilon_{\lambda}$, $\upsilon_{\varrho}$, $c_{\lambda}^j$, and $c_{\varrho}^j$, $j=0,\dots,4$, are constants which 
depend on the parameters $m$, $n$, and $r$. The  tautological maps $f_{\lambda}$ and $f_{\varrho}$ arising from $\Gamma$ are 
the timelike immersions of the torus $\mathbb{T}=\R^2/2\pi \Z^2$ 
given by
$$
 \begin{cases}
 f_{\lambda} : \mathbb{T}\ni (s,\theta)\mapsto [\cos(\theta){\rm C}_{\lambda}(s)+\sin(\theta)\Gamma'(s)]\in \EE,\\
 f_{\varrho} : \mathbb{T}\ni (s,\theta) \mapsto [\cos(\theta){\rm C}_{\varrho}(s)+\sin(\theta)\Gamma'(s)]\in \EE.
 \end{cases}
    $$
Since $f_{\lambda}(s,\theta)=-f_{\lambda}(s,\theta+\pi)$ and $f_{\varrho}(s,\theta)=-f_{\varrho}(s,\theta+\pi)$, 
the immersed surfaces 
$f_{\lambda}({\mathbb T})$ and $f_{\varrho}({\mathbb T})$ are symmetrical with respect to the wall of the AdS chambers 
${\mathcal A}^{\pm}_{{P}_4}$.  
The umbilical arcs  ${\mathcal C}_{\pm}=f_{\lambda}({\mathrm C}_{\pm})=f_{\varrho}({\mathrm C}_{\pm})$ are the intersections 
of the surfaces with the AdS wall. They are standard torus knots of type $(m,n)$  parametrized by
$$
   \R/2\pi\Z \ni   s \mapsto  \left[\pm(-\sin(ns),\cos(ns),-\sin(ms),\cos(m s),0)\right]\in {\mathcal C}_{\pm}.
      $$
If both $m$ and $n$ are odd, then ${\mathcal C}_{+}={\mathcal C}_{-}$. 
Consequently, the umbilical points are double points of the 
tautological immersions. In this case, $f_{\lambda}$ and $f_{\varrho}$ cannot be embeddings (see Figures \ref{FIG8} and \ref{FIG9}). 
If one among $m$ and $n$ is even, the umbilical arcs are disjoint. The parabolic components $f_{\lambda}(\mathbb T^{+})$ and $f_{\varrho}(\mathbb T^{-})$ belong to the positive AdS chamber, while $f_{\lambda}(\mathbb T^{-})$ and $f_{\varrho}(\mathbb T^{+})$ belong to the negative one. 
Due to the symmetry of the surfaces with respect to the wall, $f_{\lambda}(\mathbb T^{-})$ and $f_{\varrho}(\mathbb T^{+})$ are 
mirror images of  $f_{\lambda}(\mathbb T^{+})$ and $f_{\varrho}(\mathbb T^{-})$ (see Figures \ref{FIG10} and \ref{FIG11}). 
Although not homogeneous, these immersions are of cohomogeneity
one. Since the directrix curves have constant curvatures, all these immersions are second order local deformations of each other.

\bibliographystyle{amsalpha}

\end{document}